\documentclass[a4paper,11pt]{article}
\usepackage[english]{babel}
\usepackage{graphicx}

\usepackage{amsxtra}
\usepackage{amssymb}
\usepackage{amsmath}
\usepackage{amsthm}
\usepackage{color}
 \usepackage{hyperref}

\newtheorem{prop}{Proposition}[section]

\newtheorem{lemma}[prop]{Lemma}
\newtheorem{theo}[prop]{Theorem}
\newtheorem{coro}[prop]{Corollary}

\numberwithin{equation}{section}

\theoremstyle{remark}
\newtheorem{rmq}{Remark}

\newcommand{\mn}{\mathrm{n}}

\newcommand{\di}{\displaystyle}

\newcommand{\n}{\nabla}
\newcommand{\ul}{\underline}

\newcommand{\R}{\mathbb{R}}
\newcommand{\N}{\mathbb{N}}

\newcommand{\Q}{\mathbb{Q}}
\newcommand{\MP}{\mathbb{P}}
\newcommand{\tr}{\widetilde{r}}
\newcommand{\tz}{\widetilde{z}}
\newcommand{\tu}{\widetilde{u}}
\newcommand{\tw}{\widetilde{w}}

\newcommand{\ta}{\widetilde{a}}
\title{On the zero capillarity limit for the Euler-Korteweg system}
\author{Corentin Audiard
\footnote{Sorbonne Universit\'e, Universit\'e Paris Cit\'e, CNRS, Laboratoire Jacques-Louis Lions, LJLL, F-75005 Paris, France}, Marc-Antoine Vassenet\footnote{Universit\'e Paris Dauphine, PSL Research University, Ceremade, Umr Cnrs 7534, Place du
Mar\'echal De Lattre De Tassigny, 75775 Paris cedex 16, France}}
\begin{document}
\maketitle
\begin{abstract}
We study the Euler-Korteweg equations with a weak capillarity tensor. It formally converges 
to the Euler equations in the zero capillarity limit. Our aim is two-fold : first we prove 
rigorously this limit in $\R^d$, $d\geq 1$, and obtain
a more precise BKW expansion of the solution, second we initiate the study of the problem on the half space. In this 
case we obtain a priori estimates for the solutions that degenerate as the capillary coefficient 
converges to zero, and we explain
this degeneracy with the construction of a (formal) BKW expansion that exhibits boundary layers.\\
The results on the full space extend and improve a classical result of Grenier (1998) on the 
semi-classical limit of nonlinear Schr\"odinger equations.\\
The analysis on the half space is restricted to the case of quantum fluids with
irrotational velocity.
\end{abstract}
\renewcommand{\abstractname}{R\'esum\'e}
\begin{abstract}
\end{abstract}
\section{Introduction}
The Euler-Korteweg system is a modification of the compressible Euler equations that adds a capillary tensor 
in the momentum equation
\begin{equation}\label{EK}
 \left\{
 \begin{array}{ll}
  \partial_t\rho_\varepsilon+\text{div}(\rho_\varepsilon u_\varepsilon)=0,\\
  \partial_t u_\varepsilon+u_\varepsilon\cdot \nabla u_\varepsilon
  +\nabla g(\rho_\varepsilon)=\varepsilon^2
  \nabla\left(K(\rho_\varepsilon)\Delta\rho_\varepsilon+\frac{1}{2}K'(\rho_\varepsilon)
  |\nabla \rho_\varepsilon|^2\right),\\
  (\rho,u)|_{t=0}=(\rho_0,u_0).
 \end{array}
\right.\ (x,t)\in \Omega\times [0,T]
\end{equation}
The term $\varepsilon^2K$ is the capillary coefficient. We are interested in the study 
of the limit $\varepsilon\to 0$, where we recover formally the usual Euler equations. 
\begin{equation}\label{Euler}
 \left\{
 \begin{array}{ll}
  \partial_t\rho+\text{div}(\rho u)=0,\\
  \partial_t u+u\cdot \nabla u+\nabla g(\rho)=0,\\
  (\rho,u)|_{t=0}=(\rho_0,u_0).
 \end{array}
\right.\ (x,t)\in \Omega\times [0,T]
\end{equation}
We consider solutions of the form $\rho=\rho_\infty+r$, with $(r,u)\in C([0,T],H^{n+1}\times H^n)$, $n$ large, 
$\rho_\infty$ is a constant such that $g'(\rho_\infty)>0$, and 
$u\in C([0,T],H^n)$. Their (local) existence for fixed $\varepsilon$ is known since the work of
Benzoni, Danchin and Descombes \cite{Benzoni1}.
\\
Given $(\rho_0,u_0)$ smooth, we study the convergence of smooth solutions of \eqref{EK}
to the solution of \eqref{Euler}. We consider two geometric settings : $\Omega=\R^d$ and $\Omega
=\R^{d-1}\times \R^{+*}:=\R^d_+$. Precise statements are given later, but our results for  
these two cases are significantly different and can be broadly summarized as follows : 
\begin{enumerate}
 \item In the full space case, we  prove the existence of a time interval independent of 
 $\varepsilon$ on which the solutions of \eqref{EK} converge to the solution of the Euler 
 equations (``approximate solution''), with explicit rate of convergence. Moreover, thanks to 
 BKW analysis, we obtain a higher order expansion of the approximate solution with arbitrarily 
 high order of convergence.
 \item In the half space case, with boundary condition $\rho|_{x_d=0}=1,\ u_d|_{x_d=0}=0$, 
 we obtain a priori estimates of the solution that degenerate as $\varepsilon\to 0$. 
 This feature is explained by the construction of an approximate solution 
 which features terms varying rapidly in a ``boundary layer'' of size $\varepsilon$ near $x_d=0$, this 
 explains the divergence of the higher order $H^n$ 
 norms of the solution as  $\varepsilon\to 0$.
\end{enumerate}
A discussion on other choices of boundary conditions that can be found in the litterature, 
and the associated BKW expansion, is provided at the end of the article, section \ref{discussion}.
\paragraph{Link with the Schr\"odinger equation}
There is an abundant litterature on the analysis of perturbations of hyperbolic problems, the problem 
studied here has most striking similarities with the  semi-classical limit for the nonlinear 
Schr\"odinger equation 
\begin{equation}\label{schrodinger}
i\varepsilon\partial_t\psi_\varepsilon+\frac{\varepsilon^2}{2}\Delta \psi_\varepsilon=g(|\psi_\varepsilon|^2)\psi.
\end{equation}
Indeed the Madelung transform $\psi_\varepsilon=\sqrt{\rho_\varepsilon}
e^{i\varphi_\varepsilon/\varepsilon}$ allows to formally reformulate \eqref{schrodinger} as
the so-called quantum Euler system
\begin{equation}\label{qEuler}
 \left\{
 \begin{array}{ll}
  \partial_t\rho_\varepsilon+\text{div}(\rho_\varepsilon u_\varepsilon)=0,\\
  \di \partial_t u_\varepsilon+u_\varepsilon\cdot \nabla u_\varepsilon
  +\nabla g(\rho_\varepsilon)=\frac{\varepsilon^2}{4}
  \nabla\left(\frac{\Delta\rho_\varepsilon}{\rho_\varepsilon}-\frac{|\nabla \rho_\varepsilon|^2}
  {2\rho_\varepsilon^2}\right),
 \end{array}
 \right.
\end{equation}
 we recognize \eqref{EK} with $K(\rho)=1/(4\rho)$.
 \subparagraph{The Schr\"odinger equation on the full space}
 The rigorous analysis of the semi-classical limit for \eqref{schrodinger} was initiated by G\'erard 
 \cite{Gersemiclassic}, who proved the convergence to the Euler system in periodic, analytic 
 settings. This was later extended to the Sobolev framework by Grenier \cite{Grenier} thanks to a change of variable (different from the Madelung transform) which allowed to 
 reformulate  \eqref{schrodinger} as a symmetrizable hyperbolic system with a dispersive perturbation 
 \emph{which commutes with the symmetrizer}. His main result is\footnote{The exact statement in \cite{Grenier} is slightly different,  for the convenience of the reader we rephrase it in a way which is simpler for comparison in our settings.} :
\begin{theo}[Grenier '98]
  Let $\psi_\varepsilon$ solution of \eqref{schrodinger} with, for some $J\in \N$, 
  $\psi_\varepsilon|_{t=0}=
  a_0(x,\varepsilon)e^{i\varphi_0(x,\varepsilon)/\varepsilon)}$, $a_0=\sum_0^J \varepsilon^ja^j_0(x)+\varepsilon^Jr^J_\varepsilon(x)$ , $\varphi_0=\sum_{j=0}^J\varepsilon^j\varphi_0^j+\varepsilon^J\delta^J_\varepsilon$. Assume 
  $$
f'>0,\ \lim_{\varepsilon\to 0}  \|(r^J,\delta^J)\|_{H^s(\R^d)}=0\text{ for some }s>2N+2+d/2.
  $$
  Then there exists $T>0$ such that $\psi_\varepsilon$ has the form 
  $\psi_\varepsilon=a_\varepsilon e^{i\varphi_\varepsilon/\varepsilon}$ on $[0,T]\times \R^d$, 
  and there exists functions $a:=\sum_{j=0}^J\varepsilon^j a^j$ complex valued,$\ 
  \varphi=\sum_0^J\varepsilon^j\varphi^j$ defined on $[0,T]\times \R^d$ 
  given by the BKW method such that
  $$
\left\|(a_\varepsilon-a,\varphi_\varepsilon-\varphi)\right\|  _{L^\infty([0,T],H^{s-2J-2-d/2})}=o(\varepsilon^J).
$$
  \end{theo}

This fundamental result received several extensions : addition of a subquadratic potential
(Carles \cite{carlespotential}), solutions that do not cancel at infinity
(Alazard and Carles \cite{alazardcarlesGP}), a degenerate nonlinearity with $f'(0)=0$ (Alazard-Carles
\cite{alazardCarlessupercrit} with some technical limitations on the regularity, later lifted by 
Chiron and Rousset \cite{chironrousset}).
In all the results mentioned, the fluid formulation \eqref{qEuler}
is never used for the proof of convergence. Rather the authors work either directly on the Schr\"odinger
equation, or on the equations satisfied by $a_\varepsilon,\varphi_\varepsilon$, where 
$a_\varepsilon$ is \emph{complex valued}. This is a key feature since it allows to work 
on equations that have a better structure (less nonlinear, more skew-symmetric).
\subparagraph{The Schr\"odinger equation on a domain} This case is significantly more involved since the boundary conditions 
of \eqref{schrodinger} are in general not compatible with those of the limit system \eqref{Euler}. 
The construction of approximate solution through a BKW expansion then requires to add corrector 
terms that are rapidly varying, see section \ref{defbkw} for details.
When the spatial domain is the exterior of a smooth compact set in dimension $2$, Lin and Zhang \cite{linzhang} proved the convergence of the fluid variables of \eqref{schrodinger} with $g(\rho)=\rho-1$ (Gross-Pitaevskii equation), $\partial_n\psi_\varepsilon|_{\partial\Omega}=0$ (Neumann boundary condition) to the solution of the Euler equation \eqref{Euler}: 
$$
(\rho_\varepsilon,u_\varepsilon):=(
|\psi_\varepsilon|^2,\varepsilon\text{Im}(\overline{\psi_\varepsilon}\n\psi_\varepsilon))\to (\rho,\rho u),
\text{ in }L^\infty([0,T],L^2\times L^1_{\text{loc}}).$$ 
The proof is fundamentally different from the argument of Grenier as it merely uses 
a modulated energy 
$$
H_\varepsilon (\psi_\varepsilon,\rho,u)=\int_{\Omega}|(\varepsilon\n-iu)\psi_\varepsilon|^2
+(|\psi_\varepsilon|^2-\rho)^2dx,
$$
it does not extend to higher order of convergence or smoother 
functional settings. \\
When the domain is the half space $\R^{2}\times \R^{+*}$, arbitrarily precise approximate solutions 
$(\rho_{\text{app}},\varphi_{\text{app}})$ and high order of convergence were obtained by Chiron 
and Rousset \cite{chironrousset} thanks to difficult energy estimates on the error
$\di e^{-i\varphi_{\text{app}}/\varepsilon}\psi_\varepsilon-\sqrt{\rho_{\text{app}}}$. In particular, 
the skew symmetric nature of the linearized operator
$$
i\frac{\varepsilon}{2}\Delta +u_{\text{app}}\cdot\n+\frac{\text{div}(u_{\text{app}})}{2}
$$
played a key role.\\
For Dirichlet boundary conditions $\psi_\varepsilon|_{x_d=0}=0$, the analysis is even more 
difficult as the amplitude of the boundary layer terms is $O(1)$ (instead of $O(\varepsilon)$ for Neumann
conditions), Gui and Zhang managed to push further the analysis from \cite{chironrousset} to 
obtain results similar to the Neumann case, with the restriction that data are small (but with 
smallness independent of $\varepsilon$).
\paragraph{The general case of the Euler-Korteweg system} There are several reasons to study the 
Euler-Korteweg system. It includes the physically relevant Schr\"odinger equation, but more 
importantly it has also been widely considered with other capillarity coefficients. For 
example, in the framework of weak solutions, with techics similar to the modulated energy estimates, 
Bresch, Gisclon, and Lacroix-Violet \cite{BrGLV} studied the case $K(\rho)$ proportional to 
$\rho^s$, $s\in \R$, Giesselmann, Lattanzio and Tzavaras \cite{GLT} considered constant and 
general positive capillarity coefficient $K$. A conditional convergence result of weak solutions 
of the Euler-Korteweg 
system \eqref{EK} to the Euler equations \eqref{Euler} was deduced from these methods by Giesselmann 
and Tzavaras \cite{GiTz}, under ad hoc regularity assumptions on the solutions, and special 
algebraic relations for  $K$ and $g$.  To our knowledge the existence of global weak solutions
for the Euler-Korteweg system with general capillarity is still an open problem.
Other relevant capillary coefficients such as $1+\kappa/(\rho-1),\ \kappa>0$ appear in the 
framework of quasi-linear Schr\"odinger equations (e.g. \cite{KIVSHAR}).\\
When the spatial domain is an open set different from $\R^d$, another motivation to consider the fluid 
formulation of the Schr\"odinger equation is the study of the boundary value 
problem where one prescribes on the boundary the physical quantities $\rho|_{\partial\Omega}$
and $u\cdot\mathrm{n}|_{\partial\Omega}$. This is considered as the physically relevant boundary 
conditions for quantum fluids, see \cite{BNP} section $2.1$. Indeed since we have 
$\rho=|\psi|^2, u=\text{Im}(\overline{\psi}\n \psi)/|\psi|^2$, the boundary conditions on the original Schr\"odinger variable are 
highly nonlinear, and make the analysis quite difficult. The analysis on the half space is 
the subject of section \ref{sechalfspace}.\\
Finally, another important point is that the analysis of the semi-classical 
limit for the Schr\"odinger equation is restricted to the case of irrotational velocity fields : 
$\text{curl}(u)=0$. This limitation is lifted here by working directly on the fluid formulation,
to the price of more technical energy estimates.
\\
It should be noted that the convergence of solutions of \eqref{EK} to solutions of other 
models (Burgers, KdV, Kadomtsev-Petviashvili) in the long wave 
regimes was studied by Benzoni and Chiron \cite{BenChir}, the analysis relied notably on an improvement of the energy estimates introduced in \cite{Benzoni1}, quite similar to 
proposition \ref{propenergie}. 
\subparagraph{Main results}
When the spatial domain is $\R^d$, we obtain the existence of arbritrarily precise approximate 
solution, and their convergence as $\varepsilon\to 0$ to the exact solution : 
\begin{prop}[Existence of an approximate solution]\label{existapprox1}
Let $\Omega=\R^d$, $N\in \N$, $n>d/2+3+[3N/2]$, $[\cdot]$ the integer part, 
$\rho_\infty>0$, data $(r_0^k,u_0^k)\in H^{n_k}$, $0\leq k\leq N$, 
with $n_k=n-[3k/2]$, $0\leq k\leq N$, and such that $g'\circ(\rho_\infty+r^0_0)\geq c>0$. \\
There exists $\varepsilon_0>0$ and $T>0$ such that for $0<\varepsilon
<\varepsilon_0$, there exists an approximate solution \\
$(\rho^{\text{app}}-\rho_\infty,u^{\text{app}}):=\sum_0^N \varepsilon^k(r^k,u^k)$, $(r_k,u_k)$ given by the 
BKW expansion, solution of
\begin{equation*}
\left\{
\begin{array}{ll}
 \partial_t\rho+\text{div}(\rho u)=e_1,\\
 \partial_tu+u\cdot \nabla u+\nabla g=
 \varepsilon^2\nabla (K\Delta \rho+\frac{1}{2}K'|\nabla \rho|^2)+e_2,
 \\
 (\rho-\rho_\infty,u)|_{t=0}=\sum_{k=0}^N\varepsilon^k(r_0^k,u_0^k),
\end{array}
\right.
\end{equation*}
with 
\begin{eqnarray}
\|(e_1,e_2)\|_{\di C_T(H^{n_N-1}\times H^{n_N-3})}=O(\varepsilon^{N+1}),\\
\inf_{(x,t)\in \R^d\times [0,T]}\min(g'(\rho^{\text{app}}),\rho^{\text{app}}))\geq \alpha/2.
\end{eqnarray}
\end{prop}

\begin{theo}[Convergence of the approximate solution]\label{mainthRd}
Consider $(\rho_\varepsilon,u_\varepsilon)$ solution of the Euler-Korteweg system 
\begin{equation*}
\left\{
\begin{array}{ll}
 \partial_t\rho_\varepsilon+\text{div}(\rho_\varepsilon u_\varepsilon)=0,\\
 \partial_tu_\varepsilon+u_\varepsilon\cdot \nabla u_\varepsilon+\nabla g=
 \varepsilon^2\nabla (K\Delta \rho_\varepsilon+\frac{1}{2}K'|\nabla \rho_\varepsilon|^2),
 \\
 (\rho_\varepsilon-\rho_\infty,u_\varepsilon)|_{t=0}=\sum_{k=0}^N\varepsilon^k(r_0^k,u_0^k),
\end{array}
\right.
\end{equation*}
with for any $0\leq k\leq N,\ (r_0^k,u_0^k)\in H^{n-[3k/2]}$, $n>d/2+4+[3N/2]$. 
\\
Let $(\rho^{\text{app}},u^{\text{app}})\in C\big([0,T],(\rho_\infty+H^{n+1-[3N/2]})\times 
H^{n-[3N/2]}\big)$ given by  Proposition 
\ref{existapprox1}. For $\varepsilon$ 
small enough, and $n-[3N/2]-3$ even,  the exact solution $(\rho_\varepsilon,u_\varepsilon)$ 
exists on $[0,T]$ and satisfies 
$$
\|(\rho_\varepsilon-\rho^{\text{app}},u_\varepsilon-u^{\text{app}})\|_{L^\infty([0,T],H^{n-[3N/2]-3})}=O\left(\varepsilon^{N+1}\right).
$$
\end{theo}
\paragraph{Remark} The restriction ``$n-[3N/2]-3$ even'' is purely technical and related to our choice 
of energy for simplicity of the proof. The restriction can be  lifted by using 
a bit of pseudo-differential calculus as in \cite{Benzoni1}, and replaced by the sharper condition 
$n>d/2+4+[3N/2]$ with $n$  real rather than an integer.
\vspace{2mm}\\
When the spatial domain is $\R^d_+$, our results are not as complete : even the derivation 
of energy estimates requires to work in the special case $K(\rho)=1/\rho$. 
Arbitrarily precise approximate solutions in the sense of Proposition \ref{existapprox1} 
exist, but the convergence to the exact solution is still open. For consistency in 
section \ref{sechalfspace}, we construct the approximate solution in the special case $K=1/\rho$, 
but this part can be easily generalized to general $K$. On the other hand  irrotationality is 
an important simplification. For technical simplicity, we do not 
 track the precise regularity assumptions in this case (instead we work with smooth 
 functions) and we restrict the analysis to an irrotational velocity. In order to take into account the fast variation of the solution near 
the boundary, we introduce the notation $x=(x',x_d)\in \R^{d-1}\times\R^{+*}$. The use of a capital 
letter for a function $F$ generically means that it writes as 
$F(x,t)=\widetilde{F}(x',x_d/\varepsilon,t)$.\\
As usual for boundary value problems, the smoothness of the data is not enough to ensure
the smoothness of the solution, we refer to section \ref{sechalfspace} for a description of the 
additional compatibility conditions, and \ref{secfunctional} for the functional settings.
\begin{prop}[Approximate solution as a two scale expansion]\label{BKWdomain}
Assume $K(\rho)=1/\rho$. \\
Let $N\in\N$, $(r_0^k,u_0^k)_{0,\leq k\leq N}\in H^\infty(\R^d_+)$ satisfying the compatibility conditions, 
and $(u_0^k)_{0\leq k\leq N}$ irrotational. \\
There exists $\varepsilon>0$ and $T>0$ such that for $0<\varepsilon<\varepsilon_0$, there exists 
an approximate solution $(\rho^{\text{app}}-\rho_\infty,u^{\text{app}})=\sum_0^N
\varepsilon^k(r^k+R^k,u^k+U^k)$, where $R^k(x,t)=\tilde{R^k}(x',x_d/\varepsilon,t),\ U^k=\n \Phi^k,$ 
$\Phi^k(x,t)=\tilde{\Phi^k}(x',x_d/\varepsilon,t)$, and $\Phi^0=\Phi^1=0$, solution of 
\begin{equation*}
\left\{
\begin{array}{ll}
 \partial_t\rho+\text{div}(\rho u)=e_1+E_1,\\
 \partial_tu+u\cdot \nabla u+\nabla g=
 \varepsilon^2\nabla (K\Delta \rho+\frac{1}{2}K'|\nabla \rho|^2)+e_2+E_2,
 \\
 (\rho-a,u)|_{t=0}=\sum_{k=0}^N\varepsilon^k(r_0^k,u_0^k),
\end{array}
\right.
\end{equation*}
with for any $n>0$, $\|(e_1,e_2)\|_{C_TH^n}=O(\varepsilon^{N+1)},$ and 
$E_1=\tilde{E_1}(x',x_d/\varepsilon)$ 
(respectively $E_2$) is $O(\varepsilon^{N+1})$ in $\mathcal{E}_T$, respectively $O(\varepsilon^{N})$.
\end{prop}
In the rest of the paper, we shall assume $\rho_\infty=1$. This can always be done with the 
change of unknown $\rho=\rho_\infty \widetilde{\rho}$, since it preserves the 
assumption $g'(1)>0$.

\paragraph{Plan of the paper} 
Section \ref{secfunctional} is devoted to basic notations and reminder on Sobolev spaces. 
Section $3$ is focused on the proof of theorem \ref{mainthRd} : we first prove uniform 
energy estimates which imply that the solution of the Euler-Korteweg system \eqref{EK}
remains smooth on a time interval independent of $\varepsilon$, then we prove a general 
``drift estimate'' on the difference between an exact and an approximate solution. The 
construction of 
an approximate solution by BKW expansion is described in section \ref{subsecBKW}, the convergence
of the approximate solution to the exact solution is then a direct consequence of the general 
``drift estimate''.
\\
In section \ref{sechalfspace}, we initiate the analysis of the boundary value problem 
on a half space for \eqref{EK} with boundary conditions $u|_{x_d=0}\cdot e_d=0,\ \rho|_{x_d=0}=1$.
We first prove non optimal energy estimates on the solution that degenerate\footnote{We point out 
that even for $\varepsilon=1$, this is a new result.} as $\varepsilon\to 0$. Then (section 
\ref{defbkw} and after), as a possible explanation for the blow up of high $H^s$ norms in the limit $\varepsilon\to 0$, we construct a BKW expansion with boundary layer terms that are smooth 
functions of $x_d/\varepsilon$. In concluding remarks (section \ref{discussion}), we compare the 
effect on the boundary layers of other choices of boundary conditions.

\section{Notations, functional settings}\label{secfunctional}
We denote $A\lesssim B$ when there exists a constant $C$ such that $A\leq CB$. The 
possible dependance of $C$ with respect to some parameters will always be clear in 
whenever the notation is used.
\paragraph{Differential calculus}
A multi-index is generically denoted $\alpha\in \N^d$, its order is $|\alpha|=\sum_1^d\alpha_k$,
the derivative of order $\alpha$ is $\partial^\alpha:=\partial_1^{\alpha_1}\cdots 
\partial_d^{\alpha_d}$.
\\
The gradient of a vector field is the matrix 
$$
\n u=(\partial_iu_j)_{1\leq i,j\leq d}.
$$
\paragraph{Irrotational and solenoidal vector fields} We denote $\Q$, respectively $\MP$, the projector on irrotational, respectively
solenoidal, vector fields : 
$$
\Q=\Delta^{-1}\n \text{div},\ \MP=\text{I}-\Q.
$$
They are continuous self-adjoint projectors on $L^2(\R^d)$. We underline that 
$\Delta \Q$ and $\Delta \MP$ 
are differential operators, in particular for $f$ a smooth function and any $n\geq 1$,
the commutator $[\Delta^n \MP,f]$ is a differential operator of order $2n-1$, while if 
$f$ is not smooth we can use the mild estimates to bound 
$\|[\Delta^n\MP,f]g\|_2\lesssim \|f\|_{H^{2n}}\|g\|_{W^{1,\infty}}+\|f\|_{W^{1,\infty}}\|g\|_{H^{2n-1}}$.
\\
The following simple identities will be often used without mention : 
$$
\Q\n =\n ,\ \MP\n=0,\ \text{div}=\text{div}\Q,\ \text{div}\MP=0.
$$

\paragraph{Functional spaces}
For $s\in \R$ (though we will only use $s$ nonnegative integer), the $H^s$ spaces are defined as 
$$
H^s(\R^d)=\{f\in L^2:\ \int (1+|\xi|^2)^{s}|\widehat{f}(\xi)|^2d\xi<\infty\}.
$$
Of course, when $s$ is an integer, they coincide with the set of $L^2$ functions that have distributional 
derivatives in $L^2$ up to order $s$. 
Due to their even higher simplicity, we shall in particular use the $H^{2n}$ spaces for $n$ integer, that are equivalently defined as 
$$
H^{2n}(\R^d)=\{f\in L^2: \Delta^nf\in L^2\},
$$
with equivalent norm $\|f\|_2+\|\Delta^nf|_2$.\\
The space $\mathcal{E}_T$ is the set of functions smooth and exponentially decaying in 
the $x_d$ variable: 
$$
\mathcal{E}_T=\{F\in H^\infty(\R^d_+\times [0,T]), \exists\,\gamma>0:\ \forall\,j\in \N,
e^{\gamma x_d}\partial_d^jF\in L^\infty(\R_{x_d}^+,H^\infty(\R^{d-1}\times [0,T]))\}.
$$
There is no natural norm on $\mathcal{E}_T$, so in proposition \ref{BKWdomain} by $E_\varepsilon=O(\varepsilon^{N})$, we mean that
there exists $\gamma$ independent of $\varepsilon$ such that for any $n\in \N$, 
$$
\sum_{j\leq n}\|e^{\gamma x_d}\partial_d^jE_\varepsilon\|_{L^\infty(\R_{x_d}^+,H^n(\R^{d-1}\times [0,T]))}
=O(\varepsilon^N).
$$

We recall a few standard properties of Sobolev spaces (see e.g. \cite{BCD}): 
\begin{prop}\label{propsobolev}
Sobolev embedding : 
For $0<s<d/2$, $H^s\hookrightarrow L^p$, $p=\frac{2d}{d-2s}$. \\
For $s>d/2$, $s\notin d/2+\N$, $H^s\hookrightarrow C_b^{s-d/2}$. \\
Gagliardo-Nirenberg type estimates :
$$
\forall\, |\alpha|+|\beta|\leq k\in \N,\ \|D^\alpha fD^\beta g\|_2\lesssim \|f\|_\infty\|g\|_{H^k}
+\|f\|_{H^k}\|g\|_\infty.
$$
Composition rules : for $F$ smooth on some interval $I$, $F(0)=0$, $u\in H^n\cap L^\infty$, 
$n\in \N$, $\overline{\text{Im}(u)}\subset I$.
$$
\|F(u)\|_{H^n}\leq C(\|u\|_\infty)\|u\|_{H^n}.
$$
\end{prop}
A simple consequence of proposition \ref{propsobolev} is that for $|\alpha|+|\beta|=n+1$,
$n>d/2+1$, $1\leq \min(|\alpha|,|\beta|)$,
$$
\|\partial^\alpha f\partial^\beta g\|_2\lesssim \|f\|_{W^{1,\infty}}\|g\|_{H^n}
+\|f\|_{H^n}\|g\|_{W^{1,\infty}}\lesssim \|f\|_{H^n}\|g\|_{H^n}.
$$
In particular, we will frequently use the mild estimate : for $n>d/2+1$,
\begin{equation}\label{mainderiv}
\forall\,\alpha\in \N^d,\ |\alpha|=n:\ \partial^\alpha (f\n g)=f(\partial^\alpha \n g)+R,
\ \|R\|_2\lesssim \|f\|_{W^{1,\infty}}\|g\|_{H^n}+\|f\|_{H^n}\|g\|_{W^{1,\infty}}.
\end{equation}

\section{Analysis on the whole space}\label{fullspace}
\subsection{Energy estimates and the time of existence}
We remind that thanks to the change of variable $\rho\to \rho_\infty\rho$ we assume 
$\rho_\infty=1$.
Energy estimates for the Euler-Korteweg system have been derived in numerous settings, including 
the case with a small parameter \cite{BenChir}. We include here for completeness a self contained 
proof more in the spirit of \cite{Audiard9} that does not use pseudo-differential calculus.
\\
It relies on the following reformulation (due to Fr\'ed\'eric Coquel) : 
set $w=\varepsilon\sqrt{K/\rho}\nabla \rho=\n l$ 
where $l(\rho)=\varepsilon\int \sqrt{K/\rho}d\rho$, $l(1)=0$, so 
that $\n l=w$. Then 
$$
\partial_tl+u\cdot w+\varepsilon\sqrt{\rho K}\text{div}u=0.
$$
Set $a=\sqrt{\rho K}$, $z=u+iw$, after some computations
\begin{equation}\label{EKz}
\partial_tz+u\cdot \nabla z+i\nabla z\cdot w+g'(\rho)\nabla \rho+i\varepsilon\nabla(a\text{div}z)=0.
\end{equation}
The hierarchy of modified energies is the following 
$$
\forall\,n\in\N,\ E_n(\rho,u)=\frac{1}{2}\int a^{2n}\rho|\Q\Delta ^nz|^2
+|\MP (\theta_n\Delta^n z|)^2
+g'a^{2n}|\Delta^n(\rho-1)|^2dx,
$$
where $\theta_n$ is a function of $\rho$ such that 
$\frac{1}{2}((\theta_n)^2)'=\frac{\varphi_n}{a}\sqrt{K/\rho}$, that we will choose positive on a suitable interval. Heuristically, the first two terms in $E_n$ control $(\n\rho,u)$ in $H^{2n}$, 
but this control degenerates as $\varepsilon\to 0$, which is why we incorporate the third lower
order term.\\
The weights $a^{2n}\rho$ be guessed from the case $n=0$, where the conserved energy is 
$\int \rho|z|^2/2+G(\rho)dx$, and at first order $G(\rho)=g'(\rho)(\rho-1)^2/2$. For 
$n>0$, one adds the weight $a^n$ for commutation with the differential operator 
$\n(a\text{div}\cdot )$. 
\\
The following lemma is elementary, a proof can be found in \cite{Audiard9}, end of appendix A:
\begin{lemma}\label{equivnorme}
 Assume there exists some $\alpha>0$ such that $\text{Im}(\rho )\subset I\subset [\alpha,1/\alpha[$, 
 $\|\rho\|_{W^{1,\infty}}\leq \alpha^{-1}$ and
 $\forall\, s \in I,\ g'(s)\geq \alpha>0$. Choose $\theta_n$ such that $\theta_n|_{I}\geq \alpha$, 
 then there exists $c_\alpha,C_\alpha>0$ such that
 $$
c_\alpha(\|z\|_{H^{2n}}^2+\|\rho-1\|_{H^{2n}})^2\leq  E_0+E_n\leq C_\alpha
(\|z\|_{H^{2n}}^2+\|\rho-1\|_{H^{2n}})^2,
 $$
\end{lemma}
\noindent
In what follows, we denote $E_n(t)=E_n(\rho(t),u(t))$, $(\rho,u)$ solution of the Euler-Korteweg 
system.
\\
For conciseness in the computations, we define 
\begin{equation}\label{weights}
\varphi_n(\rho):=a^{2n}\rho,\ \psi_n(\rho):=a^{2n}g'(\rho).
\end{equation}
\begin{prop}\label{propenergie}
Let $(\rho_0-1,u_0)\in H^{2n+1}\times H^{2n}$, $2n>d/2+1$, $(\rho-1,u)\in 
C([0,T[,H^{2n+1}\times H^{2n})$ the unique local solution of \eqref{EK}.
\\
Assume that for some $T'\leq T$, $\alpha>0$, $\rho(\R^d\times [0,T'])\subset I\subset [\alpha,
\infty[$,
such that $g'|_{I}\geq \alpha$, then we have for $t\in [0,T'[$, $k\leq n$
\begin{equation}\label{energie}
\frac{dE_k}{dt}\leq C(\|u\|_{\infty},\|\rho\|_{W^{1,\infty}},\alpha)
(\|z\|_{H^{2k}}+\|\rho-1\|_{H^{2k}})^2(\|u\|_{W^{1,\infty}}+\|\rho\|_{W^{2,\infty}}).
\end{equation}
In particular the following holds : 
\begin{enumerate}
 \item If $\text{Im}(\rho_0)\subset I'\subset [2\alpha,\infty]$, such that on $I'$, $g'\geq 2\alpha$, 
 then the solution exists on some time interval $[0,T_1]$, $T_1$ independent of 
 $\varepsilon$, $(\rho-1,u)\in C_{T_1}(H^{2n}\times H^{2n})$ with bounds independent of 
 $\varepsilon$, and there exists an interval $I_1\subset [\alpha,\infty[$ on which $g'\geq \alpha$ such that 
 $\rho([0,T_1]\times \R^d)\subset I_1$, .
 \item In the limit $\|(z_0,\rho_0-1)\|_{H^{2n}\times H^{2n}}\to 0$, there exists $c>0$ such that 
 the time of existence is bounded from  below by $c/\|(z_0,\rho_0-1)\|_{H^{2n}\times H^{2n}}$.
\end{enumerate}

\end{prop}
\begin{proof}
If the estimate \eqref{energie} is true, the other points follow from the usual 
bootstrap argument (combined with lemma \ref{equivnorme}), so we focus on \eqref{energie}. Note also that the 
existence result of Benzoni-Danchin-Descombes implies the smoothness of the solutions, hence up to a 
standard approximation argument we can assume that the solution $(\rho,u)$ is as smooth as needed 
for the computations.\\
In the following computations, $R$ denotes generically a term which is controlled by the right hand side in \eqref{energie}. We shall use very often \eqref{mainderiv} without mention, e.g. to replace
$\Q\Delta^n (u\cdot \n z)$ by $u\cdot \n(\Q \Delta^nz)$ plus terms that can be absorbed in $R$.
Let us differentiate the three terms in $E_n$.
\paragraph{Computation of $D_1:=\frac{d}{2dt}\int_{\R^d}a^{2n}\rho |\Q\Delta^nz|^2dx$}
\begin{eqnarray*}
D_1&=& 
\text{Re}\int_{\R^d}-\varphi_n\Delta^n\Q\left(u\cdot \n z+i\n z\cdot w+i\varepsilon\n(a\text{div}z)+\n g\right)\Delta^n\Q\overline{z}\,dx
\\
&=&
\text{Im}\int_{\R^d}\underbrace{\varphi_n\left(\n(\Delta^n z)\cdot w+\varepsilon \Q\Delta^n\n(a\text{div}z)\right)\cdot\Q\Delta^n\overline{z}}_{I_1}
\underbrace{-\varphi_ng'(\Q\Delta^n\n \rho)\cdot \Q\Delta^n u}_{I_2}\,dx
\\
&&-\text{Re}\int_{\R^d}\varphi_n\left(u\cdot \n (\Q\Delta^nz)\right)\cdot\Q\Delta^n\overline{z}dx +R.
\end{eqnarray*}
We have $\text{Re}((u\cdot \n \Q\Delta^nz)\cdot\Q\Delta^n\overline{z})
=u\cdot \n|Q\Delta^nz|^2$, hence with an integration by parts we include this term in $R$.
In order to bound $I_1$, we first point out that the factor $\varepsilon$ in 
$\varepsilon\Q\Delta^n\n(a\text{div}z)$ is essential, indeed it implies thanks to \eqref{mainderiv}
$$
\varepsilon\Q\Delta^n\n(a\text{div}z)=\n(a\Delta^n\text{div}z+
2n\n a\cdot \Delta^{n-1}\n \text{div}z)+\mathcal{R},\ 
\|\mathcal{R}\|_2\leq C(\alpha) \|z\|_{W^{1,\infty}}\|z\|_{H^{2n}},
$$
We now bound $I_1$:
\begin{eqnarray*}
I_1&=&\text{Im}
\int_{\R^d}\varphi_n\left(\n \Delta^n z\cdot w+\varepsilon \n\left(a\Delta^n\text{div}z+2n\n a\cdot \n\Delta^{n-1}\text{div}z\right)\right)\Q\Delta^n\overline{z}\,dx+R
\\
&=&\text{Im}
\int_{\R^d}\varphi_n\left(\n \Delta^n z\cdot w)\cdot \Q\Delta^n\overline{z}
-\varepsilon \left(a\Delta^n\text{div}z+2n\n a\cdot \Q \Delta^{n}z\right)\right)\text{div}(\varphi_n\Q\Delta^n\overline{z})\,dx+R
\\
&=&\text{Im}
\int_{\R^d}-\varphi_n\left(\Delta^n z\cdot w\right) \text{div}\Q\Delta^n\overline{z}
+\varepsilon \text{div}\Q \Delta^n z(-a\n \varphi_n+2n\varphi_n\n a)\cdot \Q\Delta^n\overline{z}\,dx.
\end{eqnarray*}
We use $-a\varphi_n'+2n\varphi_na'=-a(a^{2n}+2n\rho a^{2n-1}a')+2n\rho a^{2n}a'
=-\varphi_n\sqrt{K/\rho}$ to obtain
\begin{eqnarray*}
I_1&=&\text{Im}
\int_{\R^d}-\varphi_n\left(\Delta^n z\cdot w\right) \text{div}\Q\Delta^n\overline{z}
-\varphi_n(\text{div}\Q\Delta^n z)
(w\cdot \Q\Delta^n\overline{z})\,dx+R
\\
&=&\text{Im}
\int_{\R^d}-\varphi_n(\MP \Delta^nz\cdot w)\text{div}(\Q\Delta^n\overline{z})
\,dx+R.
\end{eqnarray*}
This contains a loss of derivatives, which will be cancelled later thanks 
to the derivative of the solenoidal term in the energy.
\paragraph{Compensation of $I_2$} Note that without further 
computation the term $I_2$ is already without loss of derivatives, but with a loss in 
$\varepsilon$. Using $\text{div}(\rho u)=\rho\text{div}(\Q u)+u\cdot \n \rho$ we find
\begin{eqnarray*}
I_2+\frac{d}{2dt}\int_{\R d}\psi_n|\Delta^n(\rho-1)|^2dx&=& 
\int -\varphi_n g'\Delta^n\n \rho\cdot \Q\Delta^nu
-\psi_n\Delta^n\rho \Delta^n(u\cdot \n \rho +\rho \text{div}\Q u)dx
\\
&=&
\int -\varphi_n g'\Delta^n\n \rho\cdot \Q\Delta^nu
+\psi_n\n\Delta^n\rho \cdot (\Q\Delta^nu)\rho\\
&&\hspace{44mm}-\psi_nu\cdot \n \Delta^n\rho\Delta^n\rho\,dx+R
\\
&=&\int (-\varphi_ng'+\psi_n\rho)\Delta^n\n \rho\cdot \Q\Delta^nu+\text{div}(\psi_n u)
\frac{|\Delta^n\rho|^2}{2}\,dx+R
\\
&=&R,
\end{eqnarray*}
indeed thanks to definition \eqref{weights}, we have $-\varphi_ng'+\psi_n\rho=0$.
\\
To summarize, 
\begin{equation}
\label{estimD1}
\frac{d}{2dt}\int_{\R^d}\varphi_n|\Q\Delta^nz|^2+\psi_n|\Delta^n(\rho-1)|^2dx=\text{Im}
\int_{\R^d}-\varphi_n(\MP \Delta^nz\cdot w)\text{div}(\Q\Delta^n\overline{z})
\,dx+R.
\end{equation}
\paragraph{Computation of $D_2=\frac{d}{2dt}\int_{\R^d}|\MP(\theta_n\Delta^nz)|^2dx$.} We use 
$\MP\n =0$,
\begin{eqnarray*}
D_2&=&-\text{Re}\int_{\R^d}\MP\left(\theta_n\Delta^n(u\cdot \n z+i\n z\cdot w+i\varepsilon 
\n( a\text{div}z)+\n g)\right)\cdot  \MP(\theta_n\Delta^n\overline{z})\,dx
\\
&=&-\text{Re}\int_{\R^d}\bigg(u\cdot \n (\MP\theta_n\Delta^n z)
+i\varepsilon\theta_n
\n(2n\n a\cdot \Delta^n\Q z+a\Delta^n\text{div}\Q z)\\
&&\hspace{8.5cm}+\theta_ng'\Delta^n\n \rho\bigg)
\cdot  \MP(\theta_n\Delta^n\overline{z})\,dx.
\end{eqnarray*}
With the usual integration by parts, we find
$$
\text{Re}\int_{\R^d}(u\cdot \n (\theta_n\Delta^nz))\cdot \MP\theta_n \Delta^n zdx=R,
$$
$$
\int_{\R^d}(\theta_n\n\Delta^nz\cdot w)\cdot \MP(\theta_n\Delta^nz)dx=
-\int_{\R^d}(\theta_n\Delta^nz\cdot w)\text{div}\MP(\theta_n\Delta^nz)dx+R=R.
$$
We deduce 
\begin{eqnarray*}
\nonumber
D_2&=&\text{Im}
\int_{\R^d}\varepsilon\theta_n\n  (2n\n a\cdot \Delta^n\Q z+a\Delta^n\text{div}\Q z)\cdot 
\MP(\theta_n\Delta^n\overline{z})dx+R
\\
&=&-\int_{\R^d}\varepsilon(2n\n a\cdot \Delta^n\Q z+a\Delta^n\text{div}\Q z)\n\theta_n\cdot 
\MP(\theta_n\Delta^n\overline{z})dx+R
\\
&=&-\int_{\R^d}\varepsilon(a\text{div}\Q \Delta^nz)\theta_n'\n \rho\cdot \MP(\theta_n\Delta^n\overline{z})dx +R.
\end{eqnarray*}
We can replace $\MP\theta_n\Delta^nz$ by $\theta_n\MP\Delta^n z$ up to terms of order $2n-1$, which 
are then absorbed in $R$ with an integration by parts, 
this leads to 
\begin{equation}
D_2=-\int_{\R^d}\varepsilon(a\text{div}\Q \Delta^nz)\theta_n\theta_n'\n \rho\cdot \MP(\Delta^n\overline{z})dx +R.
\label{estimD2}
\end{equation}
\paragraph{Conclusion} Using \eqref{estimD1} and\eqref{estimD2}we obtain 
\begin{eqnarray*}
 \frac{dE_n}{dt}&=&
 \text{Im}
\int_{\R^d}-\varphi_n(\MP \Delta^nz\cdot w)\text{div}(\Q\Delta^n\overline{z})
+a\sqrt{\frac{\rho}{K}}\theta_n\theta_n'(\MP\Delta^nz\cdot w)
\text{div}\Q \Delta^n\overline{z}\,dx+R,
\\
&=& R,
\end{eqnarray*}
indeed the definition of $\theta_n$ ensures $\theta_n\theta_n'a\sqrt{\rho/K}-\varphi_n=0$.
\end{proof}
\begin{rmq}
Note that in the limit $\varepsilon\to 0$, we recover the usual time of existence for 
quasi-linear hyperbolic equations, with blow up criterion on $\|(\rho-1,u)\|_{W^{1,\infty}}$.
\end{rmq}

\subsection{Difference estimates}
Consider a smooth approximate solution $(\rho_1,u_1)$, $\rho_1$ bounded away from $0$, satisfying 
\begin{equation}\label{approx}
 \left\{
 \begin{array}{ll}
  \partial_t\rho_1+\text{div}(\rho_1 u_1)=e_1,\\
  \partial_t u_1+u_1\cdot \nabla u_1+\nabla g(\rho_1)=\varepsilon^2
  \nabla\left(K(\rho_1)\Delta\rho_1+\frac{1}{2}K'(\rho_1)|\nabla \rho_1|^2\right)+e_2,
 \end{array}
\right.
\end{equation}
$(e_1,e_2)$ some functions assumed to be smooth , say $H^\infty$. \\
Set generically $f_1$ for a function evaluated at $\rho_1,u_1$, in particular as previously 
$l_1=l(\rho_1)$
\begin{equation}
 \left\{
 \begin{array}{ll}
\di  \partial_tl_1+u_1\cdot w_1+\varepsilon a_1\text{div} u_1=\varepsilon\sqrt{\frac{K_1}{\rho_1}}e_1,\\
\di  \partial_t u_1+u_1\cdot \nabla u_1+\nabla g(\rho_1)=\varepsilon\nabla (a_1\text{div}w_1)
+e_2,
 \end{array}
\right.
\end{equation}
Set also generically $\widetilde{f}=f-f_1$, in particular $\tu=u-u_1$, $\tr=\rho-\rho_1$, $\widetilde{w}
=w-w_1$. The  equation on $z_1$ is 
\begin{equation}
  \partial_tz_1+u_1\cdot \nabla z_1+i\nabla z_1\cdot w_1+g'_1\n \rho_1+i\varepsilon 
  \nabla (a_1\text{div}z_1)=e_3.
\end{equation}
where $e_3=e_2+i\varepsilon \nabla \left(\sqrt{\frac{K_1}{\rho_1}}e_1\right)$
the difference equations on $\tr,\widetilde{z}$ are
\begin{eqnarray}\label{diffeq}
\left\{
\begin{array}{lll}
\partial_t\tr+\text{div}(\rho\tu +\tr u_1)&=&-e_1,\\
\partial_t \tz+u\cdot \nabla \tz+i\n\tz\cdot w
+g'\n \tr+i\varepsilon \nabla (a\text{div}\tz)&=&
-\tu\cdot \n z_1-i\n z_1\cdot \tw-\widetilde{g'}\n r_1
\\
&&+i\varepsilon\n(\ta\text{div}z_1)-e_3.
\end{array}
\right.
\end{eqnarray}
In the same spirit as the previous section, we define the energies
$$
\widetilde{E_n}=\frac{1}{2}\int \rho a^{2n}|\Q\Delta ^n\tz|^2+|\MP(\theta_n\Delta^n\tz)|^2
+a^{2n}g'(\rho)|\Delta^n\tr|^2dx.
$$
The analog of lemma \ref{equivnorme} is true :
\begin{lemma}\label{equivnorme2}
 Assume there exists some $\alpha>0$ such that $\text{Im}(\rho )\subset I\subset [\alpha,1/\alpha[$, 
 $\|\rho\|_{W^{1,\infty}}\leq \alpha$ and
 $\forall\, a \in I,\ g'(a)\geq \alpha>0$. Choose $\theta_n$ such that $\theta_n|_{I}\geq \alpha$, 
 then there exists $c_\alpha,C_\alpha>0$ such that
 $$
c_\alpha(\|\tz\|_{H^{2n}}^2+\|\tr\|_{H^{2n}})^2\leq  E_0+E_n\leq C_\alpha
(\|\tz\|_{H^{2n}}^2+\|\tr\|_{H^{2n}})^2,
 $$
\end{lemma}

\begin{prop}\label{estimerreur}Let $2n>d/2+1$, $(\rho-1,u)\in C_TH(H^{2n+1}\times H^{2n})$ given by the first 
point of proposition 
\ref{propenergie}, $(\rho_1,u_1)$ an approximate solution in $C_T(H^{2n+3}\times H^{2n+2})$, 
$\rho_1\geq \alpha$, then for $k\leq n$
\begin{eqnarray}\label{energiediff}
\frac{d\widetilde{E_n}}{dt}&\leq& C(\|\tz,z_1,\varepsilon z_1\|_{H^{2n}\times H^{2n+1}\times H^{2n+2}}+\|\tr\|_{H^{2n}}+
\|r_1\|_{H^{2n}},\alpha)
\|(\|\tz\|_{H^{2n}}+\|\tr\|_{H^{2n}})^2
\nonumber
\\
&& +\|e_2\|_{H^{2n}}^2+\|(e_1,\varepsilon \n e_1)\|_{(H^{2n})^2}^2.
\end{eqnarray}
 \end{prop}
\begin{proof}  
This is a rather straightforward modification of the proof of estimate \eqref{energie}.
For conciseness we only sketch the computations for the irrotational case, $\Q z=z$ : performing similar computations 
as for energy estimates, we obtain, with 
$R$ a term that is controlled by the right hand side of \eqref{energiediff}:
\begin{eqnarray*}
\frac{d\widetilde{E_n}}{dt}&=&R-\text{Re}\int \varphi_n\Delta^n\left(\tu \cdot\n z_1+i\n z_1\cdot \tw+\widetilde{g'}\n r_1
+i\varepsilon\nabla(\ta \text{div}z_1)-e_3\right)\overline{\Delta^n\tz}
\\
&&-\int a^{2n}\rho g'\nabla \Delta^n\tr \cdot \Delta^n\tu-
\int a^{2n}g'\Delta^n\tr \Delta^n\left(\text{div}(\rho \tu +\tr u_1)-e_1\right).
\end{eqnarray*}
A first observation is that all terms of the first line can easily be absorbed in $R$, 
for example $\text{Re}\int i\varepsilon \varphi_n\Delta^n \n(\widetilde{a}\text{div}z_1)\cdot \widetilde{z}$
is roughly bounded by $C\|\tr\|_{L^\infty}\|\varepsilon z_1\|_{H^{2n+2}}\|z\|_{H^{2n}}$.
\\
For the second line, using  
$\Delta^n\tr\n\Delta^n\tr=\frac{1}{2}\nabla (\Delta^n\tr)^2$ and integration by parts, we find
\begin{eqnarray*}
 \int a^{2n}\rho g'\nabla \Delta^n\tr \cdot \Delta^n\tu
+a^{2n}g'\Delta^n\tr \Delta^n\text{div}(\rho \tu +\tr u_1)
&=&R+\int a^{2n}\rho g'\n\Delta^n\tr\cdot \tu \Delta^n\rho\\
&&+a^{2n}g'\Delta^n\tr \left(u_1\cdot \nabla \Delta^n\tr
+(\Delta^n\text{div}u_1) \tr \right)
\\
&=&R+\int a^{2n}\rho g'\nabla \Delta^n\tr\cdot \tu \Delta^nr_1,
\end{eqnarray*}
once again this last term is taken care of with an integration by part.
\end{proof}

\subsection{BKW analysis and convergence}\label{subsecBKW}
This part can be done exactly as in previous works on the Schr\"odinger equation (see for example 
\cite{chironrousset} section $3.2$, or \cite{Grenier}), so we only recall the basic facts. \\
Write formally $r=\rho-1,\ \rho=1+r^0+\sum_1^\infty \varepsilon^kr^k$, $u=\sum_0^\infty \varepsilon^ku^k$, and
plug this ansatz in \eqref{EK}. We obtain  at rank $0$ and $1$ 
\begin{equation}\label{euler}
\left\{
\begin{array}{ll}
 \partial_t\rho^0+\text{div}(\rho^0 u^0)=0,\\
 \partial_tu^0+u^0\cdot \n u^0+\n \left(g(\rho^0)\right)=0,
\end{array}
\right.
\left\{
\begin{array}{ll}
 \partial_t r^1+\text{div}(r^1 u^0+\rho^0u_1)=0,\\
 \partial_tu^1+u^1\cdot \n u^0+u^0\cdot \nabla u^1+\n \left(g'(\rho^0)r^1\right)=0,
\end{array}
\right. 
\end{equation}
and generically the equation  or rank $k$ is the linearization of the equation at order $0$ with some 
source terms depending on the lower order terms $(r^j,u^j)_{0\leq j\leq k-1}$
\begin{equation}\label{lineuler}
\left\{
\begin{array}{ll}
\partial_t r^k+\text{div}(r^ku_0+\rho^0u^k)=f^k_1,\\
\partial_tu^k+u^0\cdot\n u^k+u^k\cdot \n u^0+\n \left(g'(\rho^0)r^k\right)=f^k_2.
\end{array}
\right.
\end{equation}
More precisely, $\di f_1^k=-\text{div}\left(\sum_{0<j,l,\ j+l=k} r^ju^l\right)$ involves derivatives of order at most $1$ of terms $(r^j,u^j)_{j\leq k-1}$, while $f_2^k$ involves similar terms
\underline{and} derivatives
up to order $3$ of terms $(r^j)_{0\leq j\leq k-2}$. It is less easy to write, but 
the main term for counting loss of derivatives is clearly $K(\rho^0)\nabla\Delta r^{k-2}$.
\\
Unsurprisingly, the system of rank $0$ is the Euler equations, that are well known to be symmetrizable (with symmetrizer $\text{diag}(g'(\rho^0)/\rho^0,1\cdots,1)$), higher order 
equations are the linearization of the Euler equations near $(\rho^0,u^0)$, with forcing terms.
The following result of well-posedness for symmetrizable hyperbolic systems is 
standard (\cite{Benzoni3}, \cite{BCD} theorem 4.15):
\begin{theo}\label{lwphyp}
 For initial data $(\rho_0^0-1,u_0^0)\in H^n,\ n>d/2+1$, 
 $\text{Im}(\rho^0)\subset I\subset [2\alpha,\infty[$, with $g'|_I\geq 2\alpha>0$. 
 There exists a time $T(\rho_0^0,u_0^0)$ such that system \eqref{euler} has a unique solution 
 in $\cap_{j=0}^n C^j([0,T],H^{n-j})$, with $\inf_{[0,T]\times \R^d}\min(\rho^0,\,g'(\rho^0))\geq \alpha$.
 \\
 For any $k\geq 1$, data $(r_0^k,u_0^k)\in H^n(\R^d)$, $n> p > d/2+1$ and forcing terms 
 $(f_1^k,f_2^k)\in H^p([0,T]\times \R^d)$, the system \eqref{lineuler} has a unique solution in $\cap_{j=0}^p
 C^j([0,T],H^{p-j})$, with $T$ the time of existence of $(\rho^0,u^0)$, it satisfies
 $$
 \|(r^k,u^k)\|_{\cap_{j=0}^p C^j([0,T],H^{p-j})}\leq C(\|(r^0_0,u^0_0)\|_{H^p},\alpha)\left(
 \|(r^k_0,u^k_0)\|_{H^n}+\|(f_1^k,f_2^k)\|_{\cap_{j=0}^p C^jH^{p-j}}\right).
 $$
\end{theo}
As a consequence of this result and composition/product rules in Sobolev spaces, we may now 
state a precise version of proposition \ref{existapprox1}:
\begin{coro}\label{existapprox}
Let $N\in \N$. For $n>d/2+1+[3N/2]$, $[\cdot]$ the integer part, 
data $(r_0^k,u_0^k)\in H^{n_k}$, $0\leq k\leq N$, 
with $n_k=n-[3k/2]$, $0\leq k\leq N$, 
there exists solutions of the 
systems \eqref{euler},\eqref{lineuler} up to order $N$, with 
$(r^k,u^k)\in \cap_{j=0}^{n_k}C^j([0,T], H^{n_k-j})$.\\
If moreover $n>d/2+3+[3N/2]$, there exists $\varepsilon_0>0$ such that for $0<\varepsilon
<\varepsilon_0$ the function $(\rho^{\text{app}},u^{\text{app}}):=(\rho^0,u^0)+\sum_1^N \varepsilon^k(r^k,u^k)$ is an 
approximate solution of the Euler-Korteweg system as in \eqref{approx} with 
\begin{eqnarray}
\|(e_1,e_2)\|_{\di C_T(H^{n_N-1}\times H^{n_N-3})}=O(\varepsilon^{N+1}),\\
\label{borneinfroa}\inf_{(x,t)\in \R^d\times [0,T]}\min(g'(\rho^{\text{app}}),\rho^{\text{app}}))\geq \alpha/2.
\end{eqnarray}
\end{coro}
\begin{proof}
The proof is an immediate application of theorem \ref{lwphyp} and composition rules in Sobolev spaces, so we 
only underline two points. First it is necessary to choose $\varepsilon$ small enough in order to ensure 
inequality \eqref{borneinfroa}.
Second, the numerology : as pointed out previously, $(f_1^k,f_2^k)$ are functions that contain 
first order derivatives of $(r^{k-1},u^{k-1})$ and third order derivatives of $r^{k-2}$. 
Hence $\di (r^0,u^0)\in \cap_{j=0}^n C^j_TH^{n-j}$, $\di (r^1,u^1)\in 
\cap_{j=0}^{n-1} C^j_TH^{n-j-1}$,
 then as $f^k_2$ contains third order derivatives of $\rho^0$, we have 
$(r^2,u^2)\in \cap_{j=0}^{n-3} C^j_TH^{n-j-3}$,  the proof is ended by induction.\\
For the estimate of $(e_1,e_2)$, it suffices to observe that the worst terms in $e_1$ 
are derivatives of order one of $(\rho^N,u^N)$, while the worst term in $e_2$ is a derivative of 
third order of $\rho^N$.
\end{proof}
\begin{rmq}
 The result is a bit better for $N=0$, as $(\rho^0,u^0)$ is  an approximate solution of 
 order $2$ if $(r_0^1,u_0^1)=0$. This is in particular the case if the initial data is simply of 
 the form $(\rho_0,u_0)$.
\end{rmq}
Putting together corollary \ref{existapprox} and proposition \ref{estimerreur}, we can now prove 
the main result.
\begin{proof}[Proof of theorem \ref{mainthRd}]
We apply proposition \ref{estimerreur} on the difference $(\tr,\tu)=(\rho-\rho^{\text{app}},u-u^{\text{app}})$, and for 
 $0\leq n\leq (n_N-3)/2$. Thanks to the bounds on the approximate solution, we obtain 
 \begin{eqnarray*}
 \|\tilde{z}(t)\|_{H^{n_N-3}}^2+\|\tr\|_{H^{n_N-3}}^2&\leq& \int_0^tC(\|\tz\|_{H^{n_N-3}}+\|\tr\|_{H^{n_N-3}},\alpha)\left(\|\tz\|_{H^{n_N-3}}^2+\|\tr\|_{H^{n_N-3}}^2\right)(s)\,ds
\\
&&
+O(\varepsilon^{N+1}).
 \end{eqnarray*}
 Gronwall's lemma ensures that, as long as $\|\tz\|_{H^{n_N-3}}+\|\tr\|_{H^{n_N-3}}=O(1)$ and $g'(\rho),\rho$ are bounded away from $0$, we have
 \begin{equation}\label{boot}
 \|\tilde{z}(t)\|_{H^{n_N-3}}^2+\|\tr(t)\|_{H^{n_N-3}}^2\leq C(\alpha)\varepsilon^{N+1}.
 \end{equation}
For $\varepsilon$ small enough a standard bootstrap argument ensures that on $[0,T]$ the solution exists with 
$g'(\rho),\rho$ bounded away from $0$ and \eqref{boot} holds. In particular, 
 we have $\|\rho-\rho^{\text{app}}\|_{C_TH^{n_N-3}}=O(\varepsilon^{N+1})$, and $\|\tu\|_{C_TH^{n_N-3}}=
 \|\text{Re}(\tz)\|_{C_TH^{n_N-3}}=O(\varepsilon^{N+1})$.
\end{proof}

\section{Analysis on the half space}\label{sechalfspace}
The case of the half space is more intricate. Even for $\varepsilon=1$, there are no 
well-posedness results for the boundary value problem for the Euler-Korteweg system. Our aim 
in this section is to initiate the analysis of the problem, by first deriving a priori estimates, 
and then performing a formal BKW expansion of the -hypothetical solution- to give an  
intuition of the effect of a boundary in the limit $\varepsilon\to 0$. We restrict 
the analysis to the case of the quantum Euler equation, that is $K(\rho)=1/\rho$.
\paragraph{A reminder on compatibilty conditions} For 
boundary value problems, this is most easily done in general abstract setting : consider
a problem of the form 
$$
\left\{
\begin{array}{ll}
\partial_tU=F(U),\\
CU|_{x_d=0}=0,\\
U|_{t=0}=U_0,
\end{array}\right.
$$
where $F$ is a smooth function of $U$ and its space derivatives. $C$ is a constant 
rectangular matrix (for our problem, $U=(\rho,u^t)^t$, $C=(\mathrm{e_1},\mathrm{e_{d+1}})^t$. 
Obviously, by continuity we have
$$
0= CU|_{x_d=0,t=0}=CU_0|_{x_d=0},
$$
this is the compatibility condition of order $0$. By differentiation in time of $CU|_{x_d=0}=0$
and use of the pde, we obtain the compatibility condition of order $1$:
$0=C\partial_tU|_{x_d=t=0}=CF(U_0)|_{x_d=0}$, the sequence of higher order 
compatibility condition is obtained by iteration of the differentation in time and use of 
$\partial_tU=F$. \\
In our settings, where $F$ and $U_0$ depend on $\varepsilon$, there is a further manipulation : 
sorting by powers of $\varepsilon$, for each compatibility condition of a fixed order we 
obtain a hierarchy of conditions, for example if $U_0= \sum \varepsilon^kU^k_0$, the compatibility condition of order $0$ implies for any $k$, $CU_0^k|_{x_d=0}=0$, the hierarchy of 
compatibility conditions of order $1$ is then obtained by Taylor expansion of the relation 
$CF(\sum \varepsilon^kU_0^k)=0$, and so on.
\\
There exists non trivial data that satisfy the compatibility conditions at all orders, for example
 $(r_0,u_0)\in C_c^\infty(\R^{d-1}\times \R^{+*})$.\\

\subsection{A priori estimates on the half space}
In this section we derive a priori estimates for irrotational solutions of the Euler-Korteweg system in the half space in the special case of quantum hydrodynamics $K=1/\rho$:
\begin{equation}\label{bvpQHD}
 \left\{
 \begin{array}{ll}
  \partial_t\rho+\text{div}(\rho u)=0,\\
  \partial_tu+u\cdot\n u+\n g(\rho)=\varepsilon^2\n\left(\frac{\Delta \rho}{\rho}-
  \frac{|\n \rho|^2}{2\rho^2}\right),
  \\
  (\rho,u)|_{t=0}=(\rho_0,u),\\
  (\rho,u_3)|_{z=0}=(1,0).
 \end{array}(x',z)\in \R^{d-1}\times \R^+,\ t\geq 0.
\right.
\end{equation}
This leads to a major simplification, indeed the main order term for the reformulated system 
on $z=u+i\varepsilon\nabla \rho/\rho$ becomes linear :
\begin{equation}\label{QHD2}
\partial_tz+u\cdot \n z+i\n z \cdot w+\n g+i\varepsilon \Delta z=0,\
 \text{with }w=\frac{\varepsilon\nabla \rho}{\rho}.
\end{equation}
Nonetheless, the analysis of the boundary value problem is quite intricate : when carrying 
the energy method as in the full space, boundary terms coming from integration by parts must 
be tracked, moreover the problem becomes 
characteristic in the limit $\varepsilon\to 0$, this causes the energy estimates to be non 
uniform in $\varepsilon$. 
Let $x=(x',x_d)$. Regularity in the tangential variables $(t,x')$ is handled differently from the regularity 
in the normal variable $x_d$, accordingly we introduce the following functionals (abusively written as norms) : 
for $z$ defined on $[0,T]\times \R^{d-1}\times \R^+$,
\begin{eqnarray*}
 \|z(t)\|_{X^n}=\sum_{2\alpha_0+\sum_1^d \alpha_k\leq n}\|\partial^\alpha z(t)\|_{L^2},\\
 \|z(t)\|_{X^n_{\tan}}=\sum_{2\alpha_0+\sum_1^{d-1} \alpha_k\leq n}\|\partial^\alpha z(t)\|_{L^2}.
\end{eqnarray*}
\\
In the same spirit as the full space, we define the following energies : for any tangential 
multi-index $\alpha=(\alpha_0,\cdots,\alpha_{d-1})\in \N^d$,
$$
E_\alpha(t)=\frac{1}{2}\int_{\R^d_+}\rho|\partial^\alpha z|^2+g'(\rho)|\partial^\alpha 
\rho|^2dx.
$$
It will be used numerous times without mention that due to the boundary conditions, for any tangential derivative 
$$\partial^\alpha \rho|_{x_d=0}=0,\ \partial^\alpha u_d|_{x_d=0}=0.$$
\\
Our main result here is :
\begin{lemma}\label{estimbvp}
 If $(\rho,u)$ is a smooth, bounded away from $0$, solution of \eqref{bvpQHD}, then 
 for $n\in \N,$ $2n>d/2+1$,
 \begin{equation}\label{controltan}
\sum_{2\alpha_0+\alpha_1+\cdots+\alpha_{d-1}\leq 2n}\frac{d}{dt}E_\alpha(t)
\leq C(\|\n z\|_{\infty}+\|\n\rho\|_\infty)\|z\|_{X^{2n}}^2
\end{equation}
with $C=C(\|\rho\|_\infty+\|\rho\|_\infty^{-1},\|\n z\|_{\infty}+1)$  a continuous function.
\\
Moreover for $0\leq j\leq n$ there exists a continuous function $F_j$
\begin{equation}\label{controlenormal}
 \|\partial_d^{2j}z\|_{X^{2(n-j)}_{\text{tan}}}\lesssim \frac{F_j(\|u\|_{X^{2n}_{\text{tan}}})}
 {\varepsilon^{4n}}
\end{equation}

\end{lemma}
\begin{rmq}
Independently of the limit $\varepsilon\to 0$, the estimates are the first step toward a 
well-posedness result similar to the one from \cite{Benzoni1} in the full space case. 
We expect that standard existence methods from the field of 
quasi-linear hyperbolic boundary value problems (e.g. \cite{Benzoni3} chapter $11$: existence 
for the linearized system with a duality argument, then an iteration scheme) 
can be tracted to our settings, since the higher order dispersive part is linear, but 
a detailed proof is beyond the aim of this section.
\end{rmq}

\begin{proof}
We recall the reformulated equations \eqref{QHD2} :
$$
\partial_tz+u\cdot \n z+i\n z \cdot w+\n g+i\varepsilon\n \text{div}(z)=0.
$$
We perform the same computations as for Proposition \ref{propenergie}, but we have to check 
the cancellation of boundary terms. As a warm up, we prove the conservation of energy 
$$
\frac{d}{dt}\int_{\R^{d-1}\times \R^+}\rho|z|^2+G(\rho)dx=0,\text{ where }G'=g.
$$
Indeed, denoting $\mathrm{n}=-\mathrm{e_d}$ the outward normal 
\begin{eqnarray*}
\frac{d}{dt}\int_{\R^{d-1}\times \R^+}\frac{\rho|z|^2}{2}+G(\rho)dx&=&
\int_{\R^{d-1}\times \R^+} \frac{-\text{div}(\rho u)|z|^2}{2}\\
&&\hspace{5mm}+\text{Re}\bigg(\rho\big(-u\cdot \n z-i\n z\cdot w-\n g
-i\varepsilon\n\text{div}(z)\big)\overline{z}\bigg)
\\
&&\hspace{65mm}-g(\rho)\text{div}(\rho u)dx
\\
&=&\int_{\R^{d-1}}-\left(\frac{|z|^2}{2}+g\right)\rho u\cdot \mn+\text{Im}(\rho 
\text{div}(z)\overline{z}\cdot\mn)
\, dx'
\\
    && -\text{Im}\int_{\R^{d-1}\times \R^+}\rho (\n z\cdot w)\cdot \overline{z}
    -\varepsilon \text{div}(z)\n \rho \cdot \overline{z}\, dx.
\end{eqnarray*}
The first integral cancels, indeed $u\cdot \mathrm{n}=-u_d=0$, and 
$$
-\text{Im}\big(\rho\text{div}(z)\overline{z}\cdot \mathrm{n}\big)
=\rho \text{div}(w)u_d-\rho \text{div}(u)w_d=-\rho \text{div}(u)w_d,
$$ 
and we have from the equation of mass conservation on the boundary:
$$
0=(\partial_t\rho+\text{div}(\rho u))|_{x_d=0}=\text{div}(u)+u\cdot\n \rho=\text{div}(u).
$$
To cancel the second integral, we use $\rho w=\varepsilon \n \rho$, and the boundary conditions
 $u_d|_{x_d=0}=0$, $w_i|_{x_d=0}=0$, $1\leq i\leq d-1$:
\begin{eqnarray*}
\text{Im}\int_{\R^{d-1}\times \R^+}\rho (\n z\cdot w)\cdot \overline{z}
    -\varepsilon \text{div}(z)\n \rho \cdot \overline{z}\, dx
    &=&\text{Im}\,\varepsilon\int_{\R^{d-1}\times \R^+} \partial_iz_j\partial_j\rho\overline{z_i}
    -\partial_j\rho \overline{z_j}\partial_iz_i \, dx
    \\
    &=&-\text{Im}\int_{\R^{d-1}\times \R^+}\rho(z_jw_j\partial_i\overline{z_i}+\overline{z_j}
    w_j\partial_iz_i)\,dx
    \\
    &&-\text{Im}\int_{\R^{d-1}\times \R^+}z_j(\partial_i\partial_j\rho)\overline{z_i}\,dx\\
    && -\text{Im}\int_{\R^{d-1}} z\cdot w \overline{z_d}\,dx'
    \\
    &=&\int_{\R^{d-1}}-(u\cdot w)w_d+|w|^2u_d\,dx'=0.
\end{eqnarray*}
The higher order estimates are similar : if $\partial^{\alpha}=\partial_t^{\alpha_0}\partial_1^{\alpha_1} \cdots\partial_{d-1}^{\alpha_{d-1}}$ contains only tangential derivatives, we denote $f_\alpha:=\partial^\alpha f$ and we have 
$$
\partial_tz_\alpha+u\cdot\n z_\alpha+i(\n z_\alpha)\cdot w 
+\partial^\alpha \n g+i\varepsilon\n( \text{div}z_\alpha)=\mathcal{C},
$$
where $\mathcal{C}$ is a quadratic commutator term that contains derivatives of order at most 
$|\alpha|$ of $z$. \\
We differentiate $E_\alpha$, and denote $R$ a generic term which has a 
$L^2$ bound of the form $C(\|\rho\|_\infty+\|\rho^{-1}\|_\infty)(\|\n z\|_{\infty}+\|\n\rho\|_\infty)
\|(z(t)\|_{H^{2n}}^2$, as in the statement of the lemma. Thanks to Gagliardo-Nirenberg 
type inequality \eqref{mainderiv} we find
\begin{eqnarray*}
 \frac{d}{dt}E_\alpha(t)&=&
\int_{\R^d_+}\varepsilon\, \text{Im}\big((\n z_\alpha\cdot \n \rho)\overline{z_\alpha}-
 \text{div}(z_\alpha)\n \rho\cdot \overline{z_\alpha}\big)
                 -\rho (\partial^\alpha\n g)\cdot u_\alpha-g' \rho_\alpha
 \partial^\alpha\text{div}(\rho u)dx +R\\
 &&+\varepsilon \text{Im}\int_{\partial\R^d_+} \rho\text{div}(z_\alpha)\overline{z_\alpha}
 \cdot \mathrm{e}_ddx'.
\end{eqnarray*}
As for the conservation of energy, we have 
$\text{Im}\big( \text{div}(z_\alpha)\overline{z_\alpha}
 \cdot \mathrm{e}_d\big)=-\text{div}(u_\alpha)w_{\alpha,d}=0$, so after 
 integration by parts
 \begin{eqnarray*}
  \frac{d}{dt}E_\alpha(t)&=&
 \varepsilon\int_{\partial \R^d_+}u_\alpha\cdot \n \rho\, w_{\alpha,d}
 \,dx\\
 && -\int_{\R^d_+}\rho g'\n\rho_\alpha\cdot u_\alpha+g' \rho_\alpha
 \text{div}(\rho u_\alpha +u\rho_\alpha)dx +R
 \\
 &=&\int_{\partial\R^d_+}g'\rho_\alpha\rho u_{\alpha,d}+\int_{\R^d_+}g'u\cdot 
\frac{ \n (\rho_\alpha)^2}{2}dx+R
\\
&=&-\int_{\partial\R^d_+}g'u_d \frac{(\rho_\alpha)^2}{2}dx'+R=R.
 \end{eqnarray*}
This is \eqref{controltan}.\\
Now to control normal derivatives, we shall use the equation to prove inductively 
\begin{equation}\label{induction}
\forall\,j\leq n,\ 
\|\partial_d^{2j}z\|_{X^{2(n-j)}_{\text{tan}}}\lesssim \frac{F(\|u\|_{X^{2n}_{\text{tan}}})}
{\varepsilon^{4j}},
\end{equation}
where $F$ is a generic smooth function that cancels at $0$. \\
Denote $\Delta'=\sum_1^{d-1}\partial_i^2$ the tangential 
laplacian, we use noncharacteristicity : 
\begin{equation}\label{noncharac}
\partial_d^{2}z=-\Delta' z+i\frac{\partial_t z}{\varepsilon}
+\frac{i}{\varepsilon}(u\cdot\n z +i\n z \cdot w+\n g).
\end{equation}
To bound  $\|\partial_d^2z\|_{X^{2n-2}_{\text{tan}}}$ we crudely bound 
$\|-\Delta' z+i\partial_tz/\varepsilon\|_{X^{2n-2}_{\text{tan}}}\lesssim \|z(t)\|_{X^{2n}_{\text{tan}}}/\varepsilon$.
The nonlinear terms are estimated with Gagliardo-Nirenberg type inequalities, for conciseness
we focus on the worst term $u_d\partial_d\partial^\alpha z$, with $\partial^\alpha$ a 
tangential derivative of order $2n-2$. We use the following interpolation inequality 
\begin{equation}\label{interp}
\forall\,k\in \N^*,\ \|f'\|_{L^2(\R^+)}\lesssim \|f\|_{L^2(\R^+)}^{1-1/k}
\|f^{(k)}\|_{L^2(\R^+)}^{1/k}.
\end{equation}
The inequality is easy when the domain is $\R$ instead of $\R^+$, it is deduced from this case
by using extension operators. Applying this to $u_d\partial_dz$ we find for some $C,C_1>0$ and fixed
$(x',t)$:
\begin{eqnarray*}
\frac{1}{\varepsilon}\|u_d\partial_d\partial^\alpha z(x',\cdot,t)\|_{L^2(\R^d_+)}\leq \|u_d\|_\infty\|\partial^\alpha\partial_dz\|_{L^2}&\leq &
\frac{C}{\varepsilon}(\|z\|_{L^2}\|\partial_dz\|_{L^2})^{1/2}\|z\|_{L^2}^{1/2}\|\partial^\alpha\partial_d^2z\|_{L^2(\R^+)}^{1/2}
\\
&\leq& \frac{C_1}{\varepsilon}\|z\|_{L^2}^{5/4}\|\partial_d^2z\|_{X^{2n-2}_{\text{tan}}}^{3/4}
\\
&\leq& \frac{C_1^4}{4\varepsilon^4}\|z\|_{L^2}^5+ \frac{3}{4}\|\partial_d^2z\|_{X^{2n-2}_{\text{tan}}}.
\end{eqnarray*}
Thanks to Sobolev's embedding, $\|\|z\|_{L^2(\R^+)}^5\|_{L^2(\R^{d-1})}\lesssim 
\|z\|_{X^{2n}_{\text{tan}}}^5$.\\
We may now proceed to the induction : assume \eqref{induction} is true for $1\leq j\leq k-1$.
To estimate $\partial^{\alpha}\partial_d^{2k}z$, $\partial^\alpha$ a tangential derivative of 
order $2n-2k$, we use equation \eqref{noncharac}
and we focus on the estimate of $\|u_d\partial_d^{2k-1}\partial^{\alpha}z/\varepsilon\|$, 
$\partial^\alpha$ tangential of order $2n-2k$:
\begin{eqnarray*}
\frac{1}{\varepsilon} \|u_d\partial_d^{2k-1}\partial^{\alpha}z\|_{L^2(\R^+)}
&\lesssim & \frac{1}{\varepsilon}\|z\|_{L^2}^{1/2}\|\partial_dz\|_{L^2}^{1/2}\|\partial_d^{2k}\partial^\alpha z\|_2
^{1/2}\|\partial_d^{2k-2}\partial^\alpha z\|_2^{1/2},
\end{eqnarray*}
we deduce for any $C>0$
$$
\frac{1}{\varepsilon} \|u_d\partial_d^{2k-1}\partial^{\alpha}z\|_{L^2(\R^d_+)}
\leq \frac{C'}{\varepsilon^{2}}\|z\|_{L^\infty(\R^{d-1}, L^2(\R^+)}\|\partial_dz\|_{L^\infty L^2}
\|\partial_d^{2k-2}z\|_{X^{2n-2k}_{\text{tan}}(\R^d_+)}+\frac{\|\partial_d^{2k}z\|_{X^{2n-2k}_{\text{tan}}}}{C}.
$$
Note that from Sobolev's embedding and interpolation 
$$\|\partial_d z\|_{L^\infty L^2}\lesssim \|\partial_d^2z\|_{X^{2n-2}_{\text{tan}}}^{1/2}
\|z\|_{X^{2n}_{\text{tan}}}^{1/2}
\leq F(\|z\|_{X^{2n}_{\text{tan}}})/\varepsilon^2,$$
we conclude 
$$
\frac{1}{\varepsilon} \|u_d\partial_d^{2k-1}\partial^{\alpha}z\|_{L^2(\R^d_+)}
\leq \frac{F(\|z\|_{X^{2n}_{\text{tan}}})}{\varepsilon^{4k}}+
\frac{\|\partial_d^{2k}z\|_{X^{2n-2k}_{\text{tan}}}}{C}.
$$
Choosing $C$ large enough, we can absorb $\frac{\|\partial_d^{2k}z\|_{X^{2n-2k}_{\text{tan}}}}{C}$
in the left hand side and complete the induction.
\end{proof}
\paragraph{A rough estimate on the time of existence} The bounds from lemma 
\ref{estimbvp} require a $L^\infty$ bound on $\rho$ to be ``self closing'',  it is easily obtained
(on very short time scale) as follows : apply the method of characteristics to the equation of mass conservation :
for any $\alpha>0$, $|\inf_{\R^d_+} (\rho(t))-\inf \rho_0|+
|\sup_{\R^d_+} (\rho(t))-\sup \rho_0|\leq \alpha$ on a time interval $[0,T]$ such that 
$\int_0^t\|\text{div}u\|_\infty\,ds\leq \ln(1+\alpha/2)$ .
$$
\forall\,j>d/2-1,\ \|f\|_{L^{\infty}(\R^d_+)}\lesssim (\|f\|_{X^j_{\text{tan}}}
\|\partial_df\|_{X^j_{\text{tan}}})^{1/2}.
$$
Denoting $E_{2n,\text{tan}}=\sum_{2\alpha_0+\alpha_1+\cdots \alpha_{d-1}}E_\alpha$, we use lemma 
\ref{estimbvp} and Sobolev's embedding
\begin{eqnarray*}
\di \frac{d}{dt}E_{2n,\text{tan}}(t)&\lesssim &
 C(\rho\|_\infty+\|\rho^{-1}\|_\infty)(\|(\n z,\n \rho)\|_{X^{2n-2}_{\text{tan}}}
\|(\n \partial_d z,\n \partial_d\rho)\|_{X^{2n-2}_{\text{tan}}})^{1/2} \|z\|_{X^{2n}}^2,
\\
 \di \|\partial_d^{2j}z\|_{X^{2n-j}_{\text{tan}}}&\lesssim&\di \frac{1}{\varepsilon^{4j}}
 F(\|z\|_{X^{2n}_{\text{tan}}}), \ 1\leq j\leq n.
\end{eqnarray*}
hence there exists a continuous function $F_1$ such that
\begin{eqnarray*}
\di \frac{d}{dt}E_{2n,\text{tan}}(t)&\lesssim &
 C(\rho\|_\infty+\|\rho^{-1}\|_\infty)\frac{F_1(\|z\|_{X^{2n}_{\text{tan}}})}{\varepsilon^{8n+3}},
\\
 \di \|\partial_d^{2j}z\|_{X^{2n-j}_{\text{tan}}}&\lesssim&\di \frac{1}{\varepsilon^{4j}}
 F(\|z\|_{X^{2n}_{\text{tan}}}), \ 1\leq j\leq n.
\end{eqnarray*}
Of course, $\|z(t)\|_{X^{2n}_{\text{tan}}}^2\sim E_{2n,\text{tan}}(t)$ with constants depending on $\|\rho\|_\infty,\|1/\rho\|_\infty$. 
It is now clear that on a timescale $O(\varepsilon^{8n+3})$, the bounds are self-closing.
\\
This is not relevant in the limit $\varepsilon\to 0$, nevertheless for $\varepsilon=O(1)$ we recover 
an estimate similar to the $\R^d$ case.
\subsection{BKW expansion : notations}\label{defbkw}
The estimates from the previous section are only obtained on a very short time interval, with a rapid 
growth of the norm of derivatives in the normal direction. A common explanation is that 
in the limit $\varepsilon\to 0$, the boundary conditions of the formal limit problem are not compatible 
with the one of the original one. Here the limit problem is the Euler equation with 
non penetration boundary condition :
\begin{equation*}
 \left\{
 \begin{array}{ll}
  \partial_t\rho +\text{div}(\rho u)=0,\ (x,t)\in \R^d_+\times \R_t^+,\\
  \partial_tu+u\cdot \n u+\n g(\rho)=0,\ (x,t)\in \R^d_+\times \R_t^+\\
  u\cdot \mathrm{e}_d=0,\ (x,t)\in (\partial\R^d_+)\times \R_t^+.
 \end{array}
\right.
\end{equation*}
The solutions of this problem do not satisfy the boundary condition 
$\rho_{\partial\R^d_+\times \R^+_t}=1$, even if the initial data do, hence the presence of boundary layers is expected, leading to the growth of the Sobolev norms of the solution.
\\
It is therefore natural to consider of a formal expansion in 
$\varepsilon$ similarly to the full space case, but with the addition of correctors 
rapidly varying in $x_d$. As is common, we search an approximate solution $(\rho_a,u_a)$, 
with $u_a=\n \phi_a$  irrotational, in the form of a two scale expansion
$$
\left\{
\begin{array}{ll}
\di \rho_a(x,t)=1+\sum_0^N\varepsilon^n r^n(x,t)+\varepsilon^nR^n(x',x_d/\varepsilon,t),\\
\di \phi_a(x,t)=\sum_0^N\varepsilon^n(\phi^n(x,t)+\Phi^n(x',x_d/\varepsilon,t)).
\end{array}\right.
$$
We shall denote $\rho^0=1+r^0$.\\
The terms $(R^n,\Phi^n)$ are the so-called boundary layer terms, they will belong to the set 
$\mathcal{E}_T$, we recall its definition:
$$
\mathcal{E}_T=\{F\in H^\infty(\R^d_+\times [0,T]), \exists\,\gamma>0:\ \forall\,j\in \N,
e^{\gamma x_d}\partial_d^jF\in L^\infty(\R_{x_d}^+,H^\infty(\R^{d-1}\times [0,T]))\}.
$$
The terms $(r^n,\phi^n)$ are the interior terms. 
Since we work with the potential $\phi^n$, it is convenient to introduce the integrated version of 
\eqref{EK}
\begin{equation}\label{EKpot}
 \left\{
 \begin{array}{ll}
\di  \partial_t\rho +\text{div}(\rho \n\phi)=0,\ (x,t)\in \R^d_+\times \R_t^+,\\
\di   \partial_t\phi+|\n\phi|^2/2+g(\rho)=\varepsilon^2\left(\frac{\Delta \rho}{\rho}
-\frac{1}{2\rho^2}|\n\rho|^2\right),
  \ (x,t)\in \R^d_+\times \R_t^+.
  \end{array}
\right.
\end{equation}
In the following, we denote $\ul{f}=f|_{x_d=0}$.
\subsection{The cascade of equations for the BKW expansion}
The usual way to obtain equations for $(R^n,\Phi^n),\ (r^n,\phi^n)$ is to split the analysis between 
$x_d$ large with respect to $\varepsilon$, where the boundary layer terms are neglected and we 
have to solve evolutionary equations on $(r^n,\phi^n)$, and conversely for $x_d$ small 
we obtain ordinary differential equations on the correctors $(R^n,\Phi^n)$.
An important tool is the following observation (see Grenier-Gu\`es \cite{GrGu}) :  
for $\varphi$ smooth, 
$(a,B)\in H^\infty\times \mathcal{E}$
\begin{equation}\label{splitinnerouter}
f(a(x)+B(x',x_d/\varepsilon))
=f(a(x))+f(a(x',0)+B(x',x_d/\varepsilon)+\varepsilon R,\ R\in \mathcal{E}_T.
\end{equation}
Inserting the ansatz in \eqref{EK}, and sorting by powers in $\varepsilon$, it is readily seen that 
$\Phi^0=0$, indeed the power $\varepsilon^{-2}$ in the (integrated) momentum equation gives 
\begin{eqnarray*}
(\partial_d\Phi^0)^2=0\Rightarrow \Phi^0=0.
\end{eqnarray*}
The first equations for the inner domain are 
\begin{equation}\label{inner0}
\left\{
\begin{array}{ll}
 \partial_tr^0+\text{div}(r^0\n \phi^0)=0,\\
 \partial_t\phi^0+\frac{1}{2}|\n \phi^0|^2+g(r^0)=0,\\
 \partial_d\phi^0|_{x_d=0}=0.
\end{array}
\right.
\end{equation}
The next equations for the boundary layer are obtained using $\ul{\partial_d\phi^0}=0$
\begin{eqnarray*}
\left\{
\begin{array}{ll}
 \di (1+\ul{r^0}+R^0)\partial_d^2\Phi^1+(\partial_d\Phi^1+\ul{\partial_d\phi^0})\partial_dR^0=0,\\
\di \frac{\partial_d^2 R^0}{1+R^0+\ul{r^0}}-\frac{1}{2}\frac{(1+\partial_dR^0)^2}{(R^0+\ul{r^0})^2}=
\partial_d\Phi^1\ul{\partial_d\phi^0}+\frac{1}{2}(\partial_d\Phi^1)^2+g(1+R^0+\ul{r^0}),\\
R^0|_{x_d=0}+\ul{r^0}=0.
\end{array}
\right.
\end{eqnarray*}
\begin{equation}\label{layer0}
\Leftrightarrow 
\left\{
\begin{array}{ll}
 \di (1+\ul{r^0}+R^0)\partial_d^2\Phi^1+(\partial_d\Phi^1)\partial_dR^0=0,\\
\di \frac{\partial_d^2 R^0}{1+R^0+\ul{r^0}}-\frac{1}{2}\frac{(\partial_dR^0)^2}{(1+R^0+\ul{r^0})^2}=
\frac{1}{2}(\partial_d\Phi^1)^2+g(1+R^0+\ul{r^0}),\\
R^0|_{x_d=0}+\ul{r^0}=0,
\end{array}
\right.
\end{equation}
Similarly to the full space case, the higher order equations for the interior terms are
\begin{equation}\label{innerk}
 \forall\,k\geq 1,\ 
 \left\{
 \begin{array}{ll}
  \partial_tr^k+\text{div}(\rho^0\n \phi^k++r^k\n \phi^0)=f_1^k,\\
  \partial_t\phi^k+\n \phi^0\cdot \n \phi^k+g'(\rho^0)r^k=f_2^k,\\
  \partial_d\phi^k|_{x_d=0}+\partial_d\Phi^{k+1}|_{x_d=0}=0.
 \end{array}
\right.
\end{equation}
where $f_1^k,f_2^k$ only depend on $(\n \phi^j,r^j)_{j\leq k-1})$ and their derivatives.
\\
The higher order boundary layer equations are 
\begin{equation}\label{layerk}
 \left\{
 \begin{array}{ll}
  \partial_d((R^0+\ul{\rho^0})\partial_d\Phi^{k+2})=F^k_1,\text{ mass, order }\varepsilon^{k},\\
\di  \frac{\partial_d^2(R^k)}{\ul{\rho^0}+R^0}-\frac{\partial_dR^k\partial_dR^0}{(\ul{\rho^0}+R^0)^2}
=g'(\ul{\rho^0}+R^0)R^k+F_2^k,\text{ momentum, order }\varepsilon^{k},
\\
\di R^k|_{x_d=0}=\ul{r^k}.
 \end{array}
\right.
\end{equation}
where $F^1_k$, respectively $F^2_k$, depends on $(\Phi^j)_{j\leq k+1},(R^j,\ul{r^j},\ul{\varphi^j})_{j\leq k}$, 
respectively $(R^j)_{j\leq k-1},(\Phi^j,\ul{r^j},\,\ul{\varphi^j})_{j\leq k}$, and are exponentially 
fast decaying. We underline here that it is important for solvability that $F_1^k$ does not depend 
on $R^{k+1}$, this is due to the fact that $\Phi^0=\Phi^1=\ul{\partial_d\varphi^0}=0$.

\subsection{Solvability of the BKW expansion}
The order in which we solve the equations is as follows
\begin{equation*}
\Phi^{k+1}\to (\varphi^{k},r^{k})\to R^{k}\to \Phi^{k+2}\cdots
\end{equation*}
The existence of the expansion will be a consequence of the following three lemmas :
\begin{lemma}[Inner expansion]\label{solvhyp}
For smooth initial data $(\rho^0_0,\varphi^0_0)\in (1+H^\infty)\times H^\infty$ that satisfy 
the compatibility conditions, there exists a time 
$T>0$ and a unique smooth solution such that $(\rho^0-1,\n\varphi^0)\in C([0,T],H^\infty)$
to \eqref{inner0}.
\\
For such $(r^0,\varphi^0)$, $T$, and $(r_0,\varphi_0)$ that satisfy the compatibility 
conditions, the linear problem 
\begin{equation}\label{lineulerbvp}
 \forall\,k\geq 1,\ 
 \left\{
 \begin{array}{ll}
  \partial_tr+\text{div}(\rho^0\n \phi+r\n \phi^0)=f_1\in H^\infty([0,T]\times \R^d_+),\\
  \partial_t\phi+\n \phi^0\cdot \n \phi+g'(\rho^0)r=f_2\in H^\infty([0,T]\times \R^d_+),\\
  \partial_d\phi|_{x_d=0}=b(x',t)\in H^\infty ([0,T]\times \R^{d-1}),\\
  (\varphi,r)|_{t=0}=(\varphi_0,r_0)\in H^\infty(\R^d_+),
 \end{array}
\right.
\end{equation}
has a unique solution with $(\n\phi,r)H^\infty([0,T]\times \R^d_+)$.
\end{lemma}
\begin{proof}
 Define $u^0=\n \varphi^0$, and take the gradient of the second equation. The new system is the 
 Euler equations with non penetration boundary condition. The existence of a smooth solution 
 is due to Schochet \cite{Schochet}. Then we obtain $\varphi^0$ simply with the formula 
 $$\varphi^0=\varphi^0(t=0)+\int_0^t|u^0|^2/2+g(\rho)ds.$$
 The system \eqref{lineulerbvp} is a hyperbolic maximal dissipative problem, the general 
 solvability can be found in \cite{Rauchsymm}, as for the smoothness of solution the method
 of proof of Schochet\footnote{The problem is characteristic, but near the boundary one can trade 
 tangential regularity to estimate $\partial_du_d,\partial_dr$, then the regularity of 
 $(u_j)_{1\leq j<d-1}$ is obtained by considering $\text{curl}(u)$, which also satisfies a dissipative 
 hyperbolic system.} works also in this case.
\end{proof}

\begin{lemma}[Boundary layer, first order]\label{bd0}
There exists $T>0$ such that
the system \eqref{layer0} has a unique solution 
 $$
 \Phi^1=0,\ R^0\in \mathcal{E}_T.
 $$
\end{lemma}
\begin{proof}
By integration of the first equation and decay at infinity, 
$(R^0+\ul{r^0})\partial_d\Phi^1=0$, hence $\Phi^1=0$ provided $R^0+\ul{r^0}\neq 0$. 
\\
The second equation rewrites 
$$
\frac{1}{\sqrt{\ul{\rho^0}+R^0}}\partial_d^2\sqrt{\ul{\rho^0}+R^0}=g(\ul{r^0}+R^0)+\ul{\partial_t\varphi^0}
+|\ul{\n\varphi^0})|^2/2=g(\ul{\rho^0}+R^0)-g(\ul{\rho^0}).
$$
Setting $A^0=\sqrt{\ul{\rho^0}+R^0}$, this rewrites 
$$
\partial_d^2A^0=A^0\big(g((A^0)^2)-g(\ul{\rho^0})\big):=f(A^0),\text{ with } 
A^0|_{x_d=0}=\ul{\rho^0}+R^0|_{x_d=0}=1.
$$
Note that $f'(\sqrt{\ul{\rho^0}})=2\ul{\rho^0}g'(\ul{\rho^0})>0$
if $\ul{\rho^0}(x',t)=1+\ul{r^0}$ is close enough to $1$, so standard ODE arguments ensure 
for any $x',t$ the existence of $A^0$ converging exponentially fast to $0$ with $(A^0)^2(x',0,t)=1$. By continuity of $\rho^0$ and the compatibility condition
$\ul{\rho^0}(x',0,0)=1$, this is true on some time interval $[0,T]$, $T$ small enough.
\end{proof}
\begin{lemma}\label{bdlk}
For $F_1,F_2$ in $\mathcal{E}_T$, the problems
\begin{equation*}
 \partial_d((R^0+\ul{\rho^0})\partial_d\Phi)=F_1,\\
\end{equation*}
and
\begin{equation*}
\left\{\begin{array}{ll}
 \di  \frac{\partial_d^2(R)}{\ul{\rho^0}+R^0}-\frac{\partial_dR\partial_dR^0}{(\ul{\rho^0}+R^0)^2}
=g'(\ul{\rho^0}+R^0)R+F_2,
\\
\di R|_{x_d=0}+\ul{r}=0.
 \end{array}
\right.
\end{equation*}
have  unique smooth solutions $R,\Phi$ in $\mathcal{E}_T$.
\end{lemma}
\begin{proof}
From the first equation we get
$$
\partial_d\Phi^k(x',x_d,t)=(R^0+\ul{\rho^0})^{-1}\int_\infty^{x_d}F_1(x',r,t)dr.
$$
Since the right hand side belongs to $\mathcal{E}_T$, another integration gives the unique 
solution in $\mathcal{E}_T$. Note that the decay at infinity of $\Phi$ does not allow to 
prescribe its value at $x_d=0$.
\\
Now $R$ satisfies :
$$
\left\{
\begin{array}{ll}
\di \partial_d\left(\frac{\partial_d R}{\ul{\rho^0}+R^0}\right)=g'(\ul{\rho^0}+R^0)R+F_2,
\\
R|_{x_d=0}+\ul{r}=1.
\end{array}
\right.
$$
This is a boundary value problem of the form $(aX')'(s)=bX(s)+F(s)$, with $a,b>0$, and $F$ exponentially 
decaying. The existence  of an 
exponentially decaying solution is a direct consequence of the change of unknown $\tilde{X}=
e^{\alpha s}X$ -with some $\alpha$ small enough-
and an application of  Lax-Milgram theorem.
\end{proof}
To summarize, from $(\Phi^j)_{j\leq k}$, $(R^j,\varphi^j,r^j)_{j\leq k-1}$, we obtain 
$\Phi^{k+1}$  by solving the first ODE in \eqref{layerk} (lemma \ref{bdlk}). Given $\Phi^{k+1}$ 
the boundary condition in \eqref{innerk} is 
well-defined so we get $(\varphi^{k},r^{k})$ with lemma \ref{solvhyp}, then from $r^{k+1}$ 
we get the boundary condition to compute $R^{k+1}$ in \eqref{layerk}, using again lemma \ref{bdlk}.

\subsection{Comparison with other boundary conditions}\label{discussion}
Previous works adressed (for the nonlinear Schr\"odinger equation) the case of Dirichlet 
boundary conditions $\varphi|_{x_d=0}=0,\rho|_{x_d=0}=1$ (Gui-Zhang \cite{GuiZhang}),
and the case of Neumann boundary conditions $\varphi|_{x_d=0}=0,\rho|_{x_d=0}=1$ (Chiron-Rousset \cite{chironrousset}). Following this terminology, we label the boundary conditions considered here 
as ``mixed Dirichlet-Neumann''
The hierarchy of corrector terms is as follows :
\begin{enumerate}
 \item Dirichlet : $R^0\neq 0,\ \Phi^1\neq 0$. Existence of the BKW expansion at any order 
 for small smooth data. 
 \item Neumann : $R^0=\Phi^1=\Phi^2=0$, $R^1\neq 0,\ \Phi^3\neq 0$. Existence of the BKW 
 expansion at any order for smooth data.
 \item Mixed boundary conditions : $\Phi^1=0$, $R^0\neq 0$, $\Phi^2\neq 0$. Existence of the BKW 
 expansion at any order for smooth data.
\end{enumerate}
The justification that $\Phi^2\neq 0$ in our case is merely a computation :  the equation satisfied by $\Phi^2$ is 
$$
\partial_d\left((\ul{\rho^0}+R^0(x_d))\partial_d\Phi^2(x_d))\right)=-\partial_d\left(R^0(x_d)
x_d\ul{\partial_d^2\varphi}\right)=F_1^2,
$$
and  $F_1^2$ is not zero.
\\
Hence the first boundary layer term for the velocity is small for the mixed boundary conditions, but not as 
small as in the Neumann case. 
While the mixed boundary conditions seem to lie in between Dirichlet and Neumann in term of the strength 
of the boundary layers, it has similar difficulty to Dirichlet since it contains $R^0$ as a $O(1)$
boundary layer term. More importantly, the boundary conditions have no simple expression in 
the Schr\"odinger formulation. The use of the Schr\"odinger formulation is a key point for the convergence analysis in both 
\cite{GuiZhang} and \cite{chironrousset}, and this is what prevents us so far from proving 
the convergence of the BKW expansion to the exact solution.

\bibliographystyle{plain}
\bibliography{biblio}

\begin{thebibliography}{10}

\bibitem{alazardCarlessupercrit}
Thomas Alazard and R\'emi Carles.
\newblock Supercritical geometric optics for nonlinear {S}chr\"odinger
  equations.
\newblock {\em Arch. Ration. Mech. Anal.}, 194(1):315--347, 2009.

\bibitem{alazardcarlesGP}
Thomas Alazard and R\'emi Carles.
\newblock W{KB} analysis for the {G}ross-{P}itaevskii equation with non-trivial
  boundary conditions at infinity.
\newblock {\em Ann. Inst. H. Poincar\'e{} C Anal. Non Lin\'eaire},
  26(3):959--977, 2009.

\bibitem{Audiard9}
Corentin Audiard.
\newblock On the time of existence of solutions of the {E}uler-{K}orteweg
  system.
\newblock {\em Ann. Fac. Sci. Toulouse Math. (6)}, 30(5):1139--1183, 2021.

\bibitem{BCD}
Hajer Bahouri, Jean-Yves Chemin, and Rapha{\"e}l Danchin.
\newblock {\em Fourier analysis and nonlinear partial differential equations},
  volume 343 of {\em Grundlehren der Mathematischen Wissenschaften [Fundamental
  Principles of Mathematical Sciences]}.
\newblock Springer, Heidelberg, 2011.

\bibitem{Benzoni1}
S.~Benzoni-Gavage, R.~Danchin, and S.~Descombes.
\newblock On the well-posedness for the {E}uler-{K}orteweg model in several
  space dimensions.
\newblock {\em Indiana Univ. Math. J.}, 56:1499--1579, 2007.

\bibitem{BenChir}
Sylvie Benzoni-Gavage and David Chiron.
\newblock Long wave asymptotics for the {E}uler-{K}orteweg system.
\newblock {\em Rev. Mat. Iberoam.}, 34(1):245--304, 2018.

\bibitem{Benzoni3}
Sylvie Benzoni-Gavage and Denis Serre.
\newblock {\em Multidimensional hyperbolic partial differential equations}.
\newblock Oxford Mathematical Monographs. The Clarendon Press Oxford University
  Press, Oxford, 2007.
\newblock First-order systems and applications.

\bibitem{BrGLV}
Didier Bresch, Marguerite Gisclon, and Ingrid Lacroix-Violet.
\newblock On {N}avier-{S}tokes-{K}orteweg and {E}uler-{K}orteweg systems:
  application to quantum fluids models.
\newblock {\em Arch. Ration. Mech. Anal.}, 233(3):975--1025, 2019.

\bibitem{carlespotential}
R\'emi Carles.
\newblock W{KB} analysis for nonlinear {S}chr\"odinger equations with
  potential.
\newblock {\em Comm. Math. Phys.}, 269(1):195--221, 2007.

\bibitem{chironrousset}
David Chiron and Rousset Fr{\'e}d{\'e}ric.
\newblock {Geometric Optics and Boundary Layers for Nonlinear Schr{\"o}dinger
  Equations}.
\newblock {\em {Communications in Mathematical Physics}}, 288(2):503--546,
  2009.

\bibitem{Gersemiclassic}
P.~G\'erard.
\newblock Remarques sur l'analyse semi-classique de l'\'equation de
  {S}chr\"odinger non lin\'eaire.
\newblock In {\em S\'eminaire sur les \'Equations aux {D}\'eriv\'ees
  {P}artielles, 1992--1993}, pages Exp. No. XIII, 13. \'Ecole Polytech.,
  Palaiseau, 1993.

\bibitem{GLT}
Jan Giesselmann, Corrado Lattanzio, and Athanasios~E. Tzavaras.
\newblock Relative energy for the {K}orteweg theory and related {H}amiltonian
  flows in gas dynamics.
\newblock {\em Arch. Ration. Mech. Anal.}, 223(3):1427--1484, 2017.

\bibitem{GiTz}
Jan Giesselmann and Athanasios~E. Tzavaras.
\newblock Stability properties of the {E}uler-{K}orteweg system with
  nonmonotone pressures.
\newblock {\em Appl. Anal.}, 96(9):1528--1546, 2017.

\bibitem{Grenier}
E.~Grenier.
\newblock Semiclassical limit of the nonlinear {S}chr\"odinger equation in
  small time.
\newblock {\em Proc. Amer. Math. Soc.}, 126(2):523--530, 1998.

\bibitem{GrGu}
Emmanuel Grenier and Olivier Gu\`es.
\newblock Boundary layers for viscous perturbations of noncharacteristic
  quasilinear hyperbolic problems.
\newblock {\em J. Differential Equations}, 143(1):110--146, 1998.

\bibitem{GuiZhang}
Guilong Gui and Ping Zhang.
\newblock Semiclassical limit of {G}ross-{P}itaevskii equation with {D}irichlet
  boundary condition.
\newblock {\em SIAM J. Math. Anal.}, 54(1):1053--1104, 2022.

\bibitem{KIVSHAR}
Yuri~S. Kivshar and Barry Luther-Davies.
\newblock Dark optical solitons: physics and applications.
\newblock {\em Physics Reports}, 298(2):81--197, 1998.

\bibitem{linzhang}
Fanghua Lin and Ping Zhang.
\newblock Semiclassical limit of the {G}ross-{P}itaevskii equation in an
  exterior domain.
\newblock {\em Arch. Ration. Mech. Anal.}, 179(1):79--107, 2006.

\bibitem{BNP}
Chi-Tuong Pham, Caroline Nore, and Marc Étienne Brachet.
\newblock Boundary layers and emitted excitations in nonlinear schrödinger
  superflow past a disk.
\newblock {\em Physica D: Nonlinear Phenomena}, 210(3):203--226, 2005.

\bibitem{Rauchsymm}
Jeffrey Rauch.
\newblock Symmetric positive systems with boundary characteristic of constant
  multiplicity.
\newblock {\em Trans. Amer. Math. Soc.}, 291(1):167--187, 1985.

\bibitem{Schochet}
Steve Schochet.
\newblock The compressible {E}uler equations in a bounded domain: existence of
  solutions and the incompressible limit.
\newblock {\em Comm. Math. Phys.}, 104(1):49--75, 1986.

\end{thebibliography}
\end{document}